\documentclass[11pt]{amsart}

\pdfoutput=1
 
\usepackage[utf8]{inputenc}
\usepackage{graphicx}
\usepackage{amsfonts}
\usepackage{amsmath}
\usepackage{amsthm}
\usepackage{amssymb}
\usepackage{pgfplots}
\pgfplotsset{compat=1.17}
\usepackage{graphicx}
\usepackage{float}
\usepackage{upgreek}
\usepackage{url}

\usepackage{siunitx}
\usepackage{booktabs}
\usepackage{textcomp, gensymb}
\usepackage{enumitem}
\usepackage{multicol}
\usepackage{mathrsfs}
\usepackage{algpseudocode}
\usepackage{adjustbox}
\usepackage{algorithm}
\usepackage{mathtools}
\usepackage{tikz-cd}
\usepackage[font=small,labelfont=bf]{caption}
\usepackage{wrapfig}
\usepackage[margin=1.24in]{geometry}
\usepackage{datetime}

\usepackage[
    backend=biber,
    style=ieee,
    citestyle=numeric-comp,
    sorting=nyt,
    dashed=false
]{biblatex}
\usepackage{hyperref}
\usepackage{xurl}
\hypersetup{
    colorlinks=true,
    linkcolor=blue,
    citecolor=red,
    breaklinks=true
}

\newdateformat{monthyeardate}{\monthname[\THEMONTH] \THEYEAR}

\def\BspN{{(B,\| \cdot \|_B)}}

\title[Measure-operator convolutions and mixed-state Gabor multipliers]{Measure-operator convolutions and\\ applications to mixed-state Gabor multipliers}

\author{Hans G. Feichtinger}
\address{Faculty of Mathematics, University of Vienna, Oskar-Morgenstern-Platz 1, A-1090 Wien, Austria and 
\newline Acoustic Research Institute, OEAW, Vienna, AUSTRIA}
\email{hans.feichtinger@univie.ac.at}

\author{Simon Halvdansson}
\author{Franz Luef}
\address{Department of Mathematical Sciences, Norwegian University of Science and Technology, 7491 Trondheim, Norway.}
\email{franz.luef@ntnu.no}
\email{simon.halvdansson@ntnu.no}
\thanks{\textit{Note: }%An earlier version of this manuscript contained weaker results in Section 4.3 where $\tr(S) = 1$ was unnecessarily assumed. This has been remedied in the present version.
In the published version of this manuscript, we assumed in Section 4.3. that our trace class operator $S$ is normalized, $\tr(S)=1$, which imposes quite some restrictions on the statements in this section. In this updated version, we state and proof the results for non-normalized trace class operators. }

\theoremstyle{plain}
\newtheorem{theorem}{Theorem}[section]
\newtheorem*{theorem*}{Theorem}
\newtheorem{lemma}[theorem]{Lemma}
\newtheorem{proposition}[theorem]{Proposition}
\newtheorem{corollary}[theorem]{Corollary}

\theoremstyle{definition}
\newtheorem{definition}[theorem]{Definition}

\newtheorem{example}[theorem]{Example}

\theoremstyle{remark}
\newtheorem*{remark}{Remark}

\newcommand{\C}{\mathbb{C}}
\newcommand{\R}{\mathbb{R}}

\newcommand{\Z}{\mathbb{Z}}
\newcommand{\N}{\mathbb{N}}

\newcommand{\supp}{\operatorname{supp}}
\newcommand{\tr}{\operatorname{tr}}

\newcommand{\D}{\mathcal{D}}

\makeatletter
\newcommand{\vast}{\bBigg@{4}}
\newcommand{\Vast}{\bBigg@{5}}
\makeatother

\def\XXint#1#2#3{{\setbox0=\hbox{$#1{#2#3}{\int}$ }
		\vcenter{\hbox{$#2#3$ }}\kern-.6\wd0}}

\begin{document}
	\maketitle
	\begin{abstract}\vspace{-9mm}
	For the Weyl-Heisenberg group, convolutions between functions and operators were defined by Werner as a part of a framework called quantum harmonic analysis. We show how recent results by Feichtinger can be used to extend this definition to include convolutions between measures and operators. Many properties of function-operator convolutions carry over to this setting and allow us to prove novel results on the distribution of eigenvalues of mixed-state Gabor multipliers and derive a version of the Berezin-Lieb inequality for lattices. New results on the continuity of Gabor multipliers with respect to lattice parameters, masks and windows as well as their ability to approximate localization operators are also derived using this framework.
	    \vspace{6mm}
	\end{abstract}
	
	\renewcommand{\thefootnote}{\fnsymbol{footnote}}
    \renewcommand{\thefootnote}{\arabic{footnote}}

    \section{Introduction}
    In recent years the framework of quantum harmonic analysis, which was introduced by R.~Werner in \cite{Werner1984}, has been successfully applied to many problems in time-frequency analysis and operator theory \cite{fu20,Luef2019_acc, Luef2018, Luef2021}. The most central operations of the framework are the \emph{function-operator} and \emph{operator-operator} convolutions which generalize classical convolutions and basic objects in time-frequency analysis such as localization operators and Cohen's class distributions. In this paper, we will be focusing on the extension of the function-operator convolution, defined for $f \in L^1(\R^{2d})$ and a trace-class operator $S$ via the Bochner integral
    \begin{align}\label{eq:main_func_op_def}
        f\star S = \int_{\R^{2d}} f(z) \pi(z)S \pi(z)^* \,dz,
    \end{align}
    where $\pi(z)\psi(t) = M_\omega T_x \psi(t)  = e^{2\pi i \omega \cdot t} \psi(t-x)$ is a time-frequency shift by $z = (x,\omega) \in \R^{2d}$. 
    
    For certain applications in time-frequency analysis, one would like to generalize this definition to the case of measures. We will show that it suffices to define this convolution $\star$ for point-measures $\delta_z$: Namely,  $\delta_z \star S = \pi(z)S\pi(z)^*$, in order to extend it to the general case, which is given weakly by
    $$
    \mu \star S = \int_{\R^{2d}} \pi(z) S \pi(z)^*\,d\mu(z)
    $$
    for $S$ in the Schatten $p$-class of operators for $1 \leq p \leq \infty$. Notably, this implies \eqref{eq:main_func_op_def} when $\mu$ is absolutely continuous with respect to Lebesgue measure.
    
    Our definition of measure-operator convolution is based on the framework developed by Feichtinger in \cite{Feichtinger2017, Feichtinger2022}, which allows us to avoid the use of Bochner integrals mentioned in \cite{Berge2023} and thus several theorems from quantum harmonic analysis follow immediately from the construction, notably a version of Young's theorem and the Fourier multiplication theorem.
    
    With the relevant definitions and basic properties in place, we explore applications of measure-operator convolutions to the setting of lattices in Section \ref{sec:applications_to_lattices} where we are able to recover classical results as well as new results on the eigenvalues of mixed-state Gabor multipliers and a version of the Berezin-Lieb inequality. Here we find that the correct setting for doing quantum harmonic analysis is requiring that the window operator $S$ is compatible with the lattice $\Lambda$ in the sense that
    $$
    A\Vert \psi \Vert^2 \leq \sum_{\lambda \in \Lambda} Q_S(\psi)(\lambda) \leq B\Vert \psi \Vert^2 \qquad \text{ for all }\psi \in L^2(\R^d),
    $$
    where $Q_S$ is a Cohen's class distribution. This setting was originally explored by Skrettingland in \cite{Skrettingland2021} but the above-mentioned results are novel as is the approach to interactions with measure-operator convolutions.

    Finally, we treat the case of varying the lattice $\Lambda$ in Section \ref{sec:gabor_seq_conv}. In this setting, we are able to establish two notable results. First, we show that for a lattice $\alpha \Z^d \times \beta\Z^d$, the associated (suitably normalized) Gabor multipliers approximate localization operators in the trace-class norm under some conditions on the mask. Moreover, we are also able to establish a result on how Gabor multipliers continuously depend on the lattice parameters, the mask and the windows provided the mask convergence is in the Wiener amalgam space $W(C_0, \ell^1)(\R^{2d})$.
    
    \subsubsection*{Notation:}
    For $1 \leq p \leq \infty$, $\mathcal{S}^p$ denotes the Schatten $p$-class of operators on $L^2(\R^d)$ with singular values in $\ell^p$ where we use the convention that $\mathcal{S}^\infty = \mathcal{L}(L^2)$, the space of bounded linear operators on $L^2(\R^d)$. A check on an operator will denote conjugation by the parity operator $P(\psi)(t) = \psi(-t)$, i.e., $\check{T} = PTP$, while for measures, $\check{\mu}(E) = \mu(-E)$ where $-E = \{ -x : x \in E \}$. For a discrete set $E$, we will write $|E|$ for its cardinality and for a non-discrete set we will use the same notation $|E|$ for the Lebesgue measure. Given a set $\Omega \subset \R^{2d}$ and $R \in \R$, we will write $R\Omega$ for the dilated set $\{ R\omega : \omega \in \Omega \}$, $\Omega^c$ for the complement $\R^{2d} \setminus \Omega$ and $\chi_\Omega$ for the indicator function of the set $\Omega$. Norms $\Vert \cdot \Vert$ and inner products $\langle \cdot, \cdot \rangle$ without subscripts will always be understood to be the ones for $L^2(\R^d)$. The rank-one operator $\psi \otimes \phi$ is defined as $(\psi \otimes \phi) (f) = \langle f, \phi \rangle \psi$, for $\phi,\psi, f \in L^2(\R^d)$.  
    
    % $denity elementFor a group $G$, we will write $0$ for its identity.
    
    \section{Preliminaries}\label{sec:preliminaries}
    In this section we go over the relevant preliminaries on time-frequency analysis and quantum harmonic analysis which should be well-known to readers familiar with other works on quantum harmonic analysis such as \cite{Luef2018, Luef2021, Skrettingland2020, Luef2019_acc}. The last subsection, Section \ref{sec:prelim_convolutions} is solely based on the two papers \cite{Feichtinger2017, Feichtinger2022} by Feichtinger and the results are later used extensively in defining measure-operator convolutions in Section \ref{sec:def_prop_measure_operator_conv}.
    
    \subsection{Time-frequency analysis}
    We briefly introduce some of the main concepts of time-frequency analysis which will be used in Section \ref{sec:applications_to_lattices}. For a more thorough introduction, the reader is referred to \cite{grochenig_book}.
    
    \subsubsection{The short-time Fourier transform}
    Given a \emph{signal} $\psi \in L^2(\R^d)$ and a \emph{window} $\varphi \in L^2(\R^d)$, the \emph{short-time Fourier transform} (STFT) of $\psi$ with respect to $\varphi$ is defined on $\R^{2d}$
    $$
    V_\varphi \psi(z) = \langle \psi, \pi(z) \varphi \rangle = \int_{\R^d} \psi(t) \overline{\varphi(t-x)}e^{-2\pi i \omega \cdot t}\,dt, \quad z = (x,\omega). 
    $$
    The interpretation is that $V_\varphi\psi(z)$ measures the time-frequency content of $\psi$ at $z = (x,\omega)$ where $x$ is the time and  $\omega$ is the frequency.
    
    The classical \emph{Moyal's identity} shows how the STFT respects inner products and the $L^2$ energy of its constituents as
    \begin{align}\label{eq:moyal}
        \big\langle V_{\varphi_1}\psi_1, V_{\varphi_2}\psi_2 \big\rangle_{L^2(\R^{2d})} = \langle \psi_1, \psi_2 \rangle \overline{\langle \varphi_1, \varphi_2 \rangle}.
    \end{align}
    Lastly we mention that the squared modulus of the STFT, $|V_\varphi \psi|^2$, referred to as the \emph{spectrogram}, is a real-valued byproduct of the STFT which is often used in applications and is an example of a \emph{quadratic} time-frequency distribution.
    
    \subsubsection{Cohen's class of quadratic time-frequency distributions}\label{sec:prelim_cohen_tf}
    There is a large collection of different quadratic time-frequency distributions and those which satisfy some basic desirable properties can be characterized as being of the form
    $$
    Q_\Phi(\psi, \phi) = W(\psi, \phi) * \Phi,\qquad W(\psi,\phi)(z) = \int_{\R^d} \psi(t+x/2) \overline{\phi(t-x/2)}e^{-2\pi i \omega \cdot t}\,dt,
    $$
    where $W$ is the (cross) Wigner distribution and $\Phi$ is a tempered distribution. Such bilinear distributions are said to belong to \emph{Cohen's class of quadratic time-frequency distributions}.
    
    \subsubsection{Localization operators}
    A non-trivial consequence of Moyal's identity \eqref{eq:moyal} is that any function $\psi \in L^2(\R^d)$ can be reconstructed from its STFT $V_\varphi \psi$ as a weak integral: 
    \begin{align}\label{eq:reconstruction_classic}
        \psi = \int_{\R^{2d}} V_\varphi \psi(z) \pi(z)\varphi\,dz.
    \end{align}
    Localization operators weight this reconstruction by some function $m : \R^{2d} \to \C$, sometimes taken to be an indicator function, as
    \begin{align}\label{eq:def_loc_op}
    	A_m^\varphi \psi = \int_{\R^{2d}} m(z) V_\varphi \psi(z) \pi(z)\varphi\,dz
    \end{align}  
    and we write $A_{\chi_\Omega}^\varphi = A_\Omega^\varphi$ when $m = \chi_\Omega$. The interpretation is that $A_m^\varphi \psi$ should be approximately concentrated in $\supp(m)$ in phase space. These operators have been extensively studied due to applications in physics, signal processing and the theory of pseudodifferential operators \cite{deGosson2011, Strohmer2006, Grchenig2011} and were first introduced in this form by I. Daubechies in \cite{daubechies1988_loc}.
    \subsubsection{Gabor frames}\label{sec:prelim_frame_theory}
    In applications, instead of considering the continuous function $V_\varphi\psi$ on $\R^{2d}$ we often use the discrete analogue $(V_\varphi \psi(\lambda))_{\lambda \in \Lambda}$ where $\Lambda$ is a lattice on $\R^{2d}$, i.e., $\Lambda = A\Z^{2d}, A \in GL(2d, \R)$. This makes computations feasible but we lose a few of the nice results we have established in the continuous setting. Of course, we want as many properties as possible to carry over from the continuous case to this setting. It turns out that the correct requirement for this is the \emph{frame condition}:
    \begin{align}\label{eq:frame_condition_classical}
        A\Vert \psi \Vert^2 \leq \sum_{\lambda \in \Lambda} |\langle \psi, \pi(\lambda) \varphi \rangle|^2 \leq B \Vert \psi \Vert^2 \qquad \mbox{for all }  \psi \in L^2(\R^d), 
    \end{align}
        where $A$ and $B$ are finite real constants. If \eqref{eq:frame_condition_classical} holds, the collection $\{ \pi(\lambda)\varphi \}_{\lambda \in \Lambda}$ is said to be a \emph{Gabor frame with frame bounds }$A, B$. Important special cases are tight frames, meaning that $A = B$, or equivalently that the frame operator, i.e. the mapping 
    \begin{align}\label{eq:frame_identity_reproducing}
        \psi \mapsto \sum_{\lambda \in \Lambda} V_\varphi \psi(\lambda)\pi(\lambda) \varphi
    \end{align}
    is a scalar multiple of the identity, analogous to the reconstruction formula \eqref{eq:reconstruction_classic} in the continuous case.

    \subsubsection{Gabor multipliers}
    As for  localization operators, \eqref{eq:frame_identity_reproducing} suggests inserting weights before reconstruction in the tight case. More generally, if the frame condition \eqref{eq:frame_condition_classical} is satisfied, there exists a \emph{dual window} $\gamma$ such that we have a reconstruction of $\psi$ using an unconditionally convergent series expansion in $ L^2(\R^{d})$:
    $$
    \psi = \sum_{\lambda \in \Lambda} V_\varphi \psi(\lambda) \pi(\lambda) \gamma\qquad \text{ for all }\psi \in L^2(\R^{d}).
    $$
    {\it Gabor multipliers} can then be defined with the help of a {\it mask} or {\it (upper) symbol}, i.e., a weight function $m: \Lambda \to \C$, as follows
    $$
    G^{\varphi, \gamma}_{m, \Lambda} \psi = \sum_{\lambda \in \Lambda} m(\lambda) V_\varphi \psi(\lambda) \pi(\lambda) \gamma.
    $$
    
    \subsubsection{Wiener amalgam spaces and bounded uniform partitions of unity}\label{sec:wiener_and_bupu}    
    In order to describe  the global behavior of a local property of a function  \emph{Wiener amalgam spaces} \cite{he03, fe83, fe77-3} are a natural choice. For $1 \leq p,q < \infty$ the original construction on $\R$ selects measurable functions such that the following norm
    is finite: 
    \begin{align}\label{eq:wiener_og}
        \Vert F \Vert_{W(L^p, \ell^q)} = \left(\sum_{n \in \Z} \left(\int_n^{n+1} |F(t)|^p\,dt\right)^{q/p} \right)^{1/q}.
    \end{align}
     Here $p$ controls the local behavior of the function while $q$ controls global summability. There are many generalizations of Wiener amalgam spaces. We will make use of the characterization 
     of these space (on $\R^{2d}$)  using \emph{bounded uniform partitions of unity} (BUPU's). Given an index set $I$, a BUPU is a collection  of  (continuous) functions $\Psi = \{ \psi_i \}_{i \in I}$ such that  
    \begin{enumerate}[label=(\roman*)]
        \item $\sum_{i \in I} \psi_i(x)= 1, \quad x \in \R^{2d}$, 
        \item $\sup_{i \in I} \Vert \psi_i \Vert_{L^\infty} < \infty$,
        \item There exists a compact set $U \subset \R^{2d}$ with nonempty interior and points $z_i$ such that $\supp(\psi_i) \subset U + z_i$ for all $i \in I$,
        \item For each compact $K \subset \R^{2d}$,
        $$
         \sup_{i \in I} \big|\big\{ j \in I : K + z_i \cap K + z_j \neq \emptyset \big\}\big| < \infty,
        $$
        this is sometimes called the \emph{bounded overlap property}.
    \end{enumerate} 
    A trivial construction of a BUPU is the collection of indicator functions of translated fundamental domains of a lattice $\Lambda = \alpha \Z^d \times \beta \Z^d$, $\psi_{n,k} = \chi_{[\alpha n, \alpha (n+1)) \times [\beta k, \beta (k+1))}$ or better more intricate constructions with continuous or smooth functions, e.g. B-splines. 

    For our generalization of the Wiener amalgam spaces considered above, we will consider functions on $\R^{2d}$ which are locally in some Banach space $(B,\Vert.\Vert_B)$. The norm can then be written as
    \begin{align*}
        \Vert F \Vert_{W(B, \ell^q)} = \left( \sum_{i \in I} \Vert F \psi_i \Vert_B^q \right)^{1/q}.
    \end{align*}
    A central  property of these  norms  to be used  later is the fact  that different BUPU's $\Psi$ give equivalent norms, \cite{he03}. 
    For $p= \infty$ obvious modifications to \eqref{eq:wiener_og} apply. %hgfei 
    
    \subsection{Operator theory}
    The main tool from operator theory will be  the \emph{singular value decomposition} of a compact operator $A$, which can be written as
    $$
    A = \sum_n s_n(A) (\psi_n \otimes \phi_n),
    $$
    where $\{ \psi_n \}_n$ and $\{ \phi_n \}_n$ are orthonormal sequences  in $L^2(\R^d)$ and $(s_n(A))_n$ is a sequence of non-negative, decreasing numbers, the \emph{singular values}, converging to zero. When the operator $A$ is self-adjoint, we can take $\phi_n = \psi_n$ and the singular values are the eigenvalues of $A$ (up to a sign). Otherwise they are just the eigenvalues of the positive operator $|A| = \sqrt{A^*A}$.
    
    The Schatten $p$-class of operators for $p < \infty$ is the Banach space of compact operators with singular values in $\ell^p$. In many ways these spaces behave as the $L^p$ spaces with the \emph{trace}, defined as $\tr(A) = \sum_n \langle Ae_n, e_n\rangle$ for some orthonormal basis $\{ e_n \}_n$ of $L^2(\R^d)$, replacing the integral so that $\Vert A \Vert_{\mathcal{S}^p} = \tr(|A|^p)^{1/p}$. These norms are monotonic in $p$ in the sense that if $1 \leq p \leq  p' \leq \infty$,
    \begin{align*}
        \Vert S \Vert_{\mathcal{S}^\infty} \leq \Vert S \Vert_{\mathcal{S}^{p'}} \leq \Vert S \Vert_{\mathcal{S}^p} \leq \Vert S \Vert_{\mathcal{S}^1}.
    \end{align*}
    The cases $p=1,2$ are of special relevance with the $\mathcal{S}^1$ operators being referred to as \emph{trace-class} and $\mathcal{S}^2$ being the \emph{Hilbert-Schmidt} operators. The $\mathcal{S}^p$ - $\mathcal{S}^q$ (sesquilinear) duality brackets for the Schatten classes may be given by
    \begin{equation}
    \langle A, B \rangle_{\mathcal{S}^p, \mathcal{S}^q} = \tr(AB^*).    
    \end{equation}
    For more on this class of operators, see \cite{Simon2005}. 
    By this duality the bounded linear operators form the dual
    space to $\mathcal{S}^1$. 
    
    \subsection{Quantum harmonic analysis}
    
    In this section we recall some classical properties of quantum harmonic analysis without proofs so that they can be compared to the corresponding results for measure-operator convolutions. The reader familiar with quantum harmonic analysis can safely skip over this section. For proofs, see \cite{Werner1984, Luef2018}.
    
    \subsubsection{Operator convolutions}
    A central idea in quantum harmonic analysis is that just as functions can be moved around using translations $T_x$, so can operators in phase space using a representation $\pi$, via the operator translation:
    \begin{equation}  \label{optransl} 
    \alpha_z(S) = \pi(z)S\pi(z)^*.
    \end{equation}
    Together with the idea that traces are the proper substitute for integrals in $\mathcal{S}^p$, this leads to the following definition of function-operator and operator-operator convolutions which form the key ingredients of QHA as introduced by R.~Werner: 
    \begin{align}\label{eq:op_op_conv_def}
    	f \star S = \int_{\R^{2d}} f(z)\alpha_z(S) \,dz,\qquad T \star S(z) = \tr\big(T \alpha_z(\check{S})\big).
    \end{align}
    Note in particular that $f \star S$ is an operator on $L^2(\R^d)$ while $T \star S$ is a function on $\R^{2d}$. We also define convolutions between an operator and a function as $S \star f := f \star S$. Below we collect some properties of these convolutions and later show how they relate to time-frequency analysis in Sections \ref{sec:prelim_cohen} and \ref{sec:mixed_state_loc}.
    \begin{proposition}\label{prop:op_conv_properties}
    	Let $f,g \in L^1(\R^{2d}),\, S \in \mathcal{S}^p,\, T \in \mathcal{S}^q$ for $1 \leq p,q, r \leq \infty$ with $\frac{1}{p} + \frac{1}{q} = 1 + \frac{1}{r}$ and $R$ be a compact operator, then
    	\begin{enumerate}[label=(\roman*)]
    		\item $f \star S$ is positive if $f$ is non-negative and $S$ is positive,
    		\item \label{item:op_op_positive}$T \star S$ is non-negative if $T$ and $S$ are positive,
    		\item $\tr(f \star S) = \tr(S) \int_{\R^{2d}} f(z)\,dz$ for $p=1$,
    		\item $(f \star S)^* = \bar{f} \star S^*$,
    		\item \label{item:ass1} $(f \star S) \star T = f * (S \star T)$,
    		\item \label{item:ass2} $(f * g) \star S = f  \star (g \star S)$,
    		\item \label{item:func_compact_conditon} $f \star R$ is compact,
    		\item $\Vert f \star S \Vert_{\mathcal{S}^p} \leq \Vert f \Vert_{L^1} \Vert S \Vert_{\mathcal{S}^p}$,
                \item \label{item:op_op_conv_interpol} $\Vert T \star S \Vert_{L^r} \leq \Vert S \Vert_{\mathcal{S}^p} \Vert T \Vert_{\mathcal{S}^q}$,
                \item \label{item:op_op_conv_explicit}$\int_{\R^{2d}} T \star S(z)\,dz = \tr(T)\tr(S)$ if $p=q=1$.
    	\end{enumerate}
    \end{proposition}
    We also mention that both convolutions are commutative, function-operator convolutions by definition while for operator-operator convolutions this can be manually verified.
	
    \subsubsection{Fourier transforms}
    There is also an associated Fourier transform for operators, the \emph{Fourier-Wigner transform}, which generalizes the classical convolutions theorem in two different ways. It is defined for $S \in \mathcal{S}^1$ as a function on $\R^{2d}$ given by
    \begin{align}\label{eq:fourier_wigner_def}
        \mathcal{F}_W(S)(z) = e^{-\pi i x \omega} \tr(\pi(-z) S),
    \end{align}
    where $z = (x, \omega)$. The terminology comes from Folland \cite{HAPS} but was first used for operators by Werner in \cite{Werner1984} and has since been used extensively in quantum harmonic analysis \cite{Luef2018, Luef2021}.
    
    Before stating the convolution theorems, we also need to introduce the symplectic Fourier transform which is essentially a rotated Fourier transform on $\R^{2d}$, defined as
    $$
    \mathcal{F}_\sigma(f)(z) = \int_{\R^{2d}} f(z') e^{-2\pi i \sigma(z,z')}\,dz',
    $$
    where $z = (x,\omega),\, z' = (x', \omega')$ and $\sigma(z,z') = \omega x' - \omega' x$ is the standard symplectic form. With this, we have all the notation we need for the convolution theorem in place.
    \begin{proposition}\label{prop:function_convolution_theorem}
    Let $f \in L^1(\R^{2d})$ and $S, T \in \mathcal{S}^1$, then
    \begin{align*}
        \mathcal{F}_W(f \star S) &= \mathcal{F}_\sigma(f) \mathcal{F}_W(S),\\
        \mathcal{F}_\sigma(T \star S) &= \mathcal{F}_W(T) \mathcal{F}_W(S).
    \end{align*}
    \end{proposition}

	\subsubsection{Cohen's class distributions}\label{sec:prelim_cohen}
	In \cite{Luef2019}, the Cohen's class distributions discussed in Section \ref{sec:prelim_cohen_tf} above were characterized as the bilinear maps $Q : L^2(\R^d) \times L^2(\R^d) \to L^\infty(\R^{2d})$ such that $Q = Q_S$ for some $S \in \mathcal{L}(L^2(\R^d))$ where
	$$
	Q_S(\psi, \phi)(z) = (\psi \otimes \phi) \star \check{S}(z) = \big\langle \check{S}\pi(z)^* \psi, \pi(z)^* \phi \big\rangle.
	$$
	Moreover, when $S$ is positive and trace-class, we can expand $S$ in its singular value decomposition as $S = \sum_n \lambda_n (\varphi_n \otimes \varphi_n)$ and write
	\begin{align*}
    	Q_S(\psi)(z) = \sum_n \lambda_n |\langle \psi, \pi(z)\varphi_n\rangle|^2.
	\end{align*}
	by \cite[Theorem 7.6]{Luef2019}.
	
	We also mention that there is a reconstruction of identity related to Cohen's class distributions which generalizes \eqref{eq:reconstruction_classic}, first shown in \cite[Proposition 3.2 (3)]{Werner1984}. It states that if $S$ is trace-class with $\tr(S) = 1$,
	\begin{align}\label{eq:reconstruction_cohen}
	    \psi = \int_{\R^{2d}} \pi(z)S \pi(z)^* \psi \,dz,
	\end{align}
        which in particular implies that $1 \star S = \tr(S) I_{L^2}$.

	\subsubsection{Mixed-state localization operators}\label{sec:mixed_state_loc}
	Inspired by the identity \eqref{eq:reconstruction_cohen} and the fact that the case where $S$ is a rank-one operator precisely corresponds to classical localization operators $A_m^\varphi = m \star (\varphi \otimes \varphi)$, \emph{mixed-state} localization operators were introduced in \cite{Luef2019}. These correspond to letting $S$ be an arbitrary trace-class operator, i.e., an operator of the form $f \star S$. The terminology is borrowed from physics where each rank-one term in the singular value decomposition of a trace-class operators can be seen as a \emph{pure state} while a general trace-class operator corresponds to a \emph{mixing} of these states.
	
	\subsection{A functional-analytic approach to convolutions}\label{sec:prelim_convolutions}
    In the interest of completeness, we recall the key definitions and theorems from \cite{Feichtinger2017, Feichtinger2022} which we later use to set up measure-operator convolutions and deduce their elementary properties. For the necessary background on group representations, we refer to the book by Folland \cite{Folland}.
    
    We first go through the results in \cite{Feichtinger2022}, starting with some relevant definitions.
    \begin{definition}\label{def:isometric_rep}
	A mapping $\rho : G \to \mathcal{L}(B)$ is called an \emph{isometric representation} of a group $G$ on a Banach space $B$ if the mapping $\rho$ is linear, preserves identities, is a group homomorphism, i.e., satisfies
	\begin{align*}
	    \rho(xy) &= \rho(x) \circ \rho(y) \qquad \textrm{ for all }x,y \in G,
	\end{align*}
	and if each of the operators are isometric on $B$ meaning that
	$$
	\Vert \rho(x) S \Vert_B = \Vert S \Vert_B \qquad \textrm{ for all }x \in G \textrm{ and }S \in B.
	$$
	Moreover, if the mapping $x \mapsto \rho(x)S$ is continuous from $G$ to $B$ in the sense that
	$$
	\lim_{x \to 0} \Vert \rho(x) S - S \Vert_B = 0,
	$$
	we say that the representation $\rho$ is \emph{strongly continuous}.
	\end{definition}
        In \cite{sh71}, pairings $(B, \rho)$ are called {\it abstract homogeneous Banach spaces}. The usual {\it homogeneous Banach spaces} are Banach spaces of locally integrable functions endowed with the regular representation (by translations). For us, $B$ will be the space of trace-class operatos $\mathcal{S}^1$ and $\rho(z) = \alpha_z$ will be the operator translations \eqref{optransl}.
 
    We will use the notation $\bullet_\rho : G \times B \to B$ for the action of the representation and also define the representation for point measures as
    \begin{equation}
        \delta_x \bullet_\rho S := x \bullet_\rho S =  \rho(x) S.    
    \end{equation}
    Recall that a \emph{Banach algebra} is a Banach space $(A, \Vert \cdot \Vert_A)$ together with a continuous operation $* : A \times A \to A$. We can endow a Banach algebra with the following additional structure.
    \begin{definition}
        A Banach space $(B, \Vert \cdot \Vert_B)$ is a \emph{Banach module} over a Banach algebra $(A, *, \Vert \cdot \Vert_A)$ if one has a mapping $(a, b) \mapsto a \bullet b$ from $A \times B$ to $B$ that is bilinear and associative, i.e.,
        \begin{enumerate}[label=(\roman*)]
            \item $\Vert a \bullet b \Vert_B \leq \Vert a \Vert_A \Vert b \Vert_B$,
            \item $a_1 \bullet (a_2 \bullet b) = (a_1 * a_2) \bullet b$.
        \end{enumerate}
        If moreover
        $$
        \overline{\operatorname{span}( A \bullet B)} = B,
        $$
        the Banach module is said to be \emph{essential}.
    \end{definition}
    For us, the Banach algebra $A$ will be the convolutions algebra of bounded measures later on.
    
    We are now ready to state the main theorem we will use which is a reformulation of \cite[Theorem 2]{Feichtinger2022}. In the sequel, $M(G)$ denotes the space of bounded complex signed measures on the locally compact abelian group $G$ with $\Vert \mu \Vert_{M} = |\mu|(G)$, viewed as a Banach algebra with respect to convolution, see \cite{Feichtinger2017}. 
    \begin{theorem}\label{thm:existence_convolution}
        Let $\BspN$ be a Banach space and $\rho$ a strongly continuous and isometric representation of a locally compact abelian group $G$. Then $B$ is a Banach module over the Banach algebra $(M(G), \Vert \cdot \Vert_M)$ with respect to a mapping $\bullet_\rho : (\mu, S) \mapsto \mu \bullet_\rho S$ that extends the action of the discrete measure $\delta_x \bullet_\rho S = \rho(x)S$ and satisfies the estimate
        $$
        \Vert \mu \bullet_\rho S \Vert_B \leq \Vert \mu \Vert_{M(G)} \Vert S \Vert_B,\qquad \mu \in M(G),\, S \in B.
        $$
    \end{theorem}
    We interpret such a mapping $\bullet_\rho$ as a (generalized) convolution between the measure $\mu$ and the operator $S$.
    
    \begin{remark}
    By identifying $L^1(G)$ with a closed subspace  (ideal) of $M(G)$, this also defines a convolution between functions and elements of $B$.
    \end{remark}
    
    We moreover have a uniqueness theorem for this extension which is \cite[Theorem 6]{Feichtinger2022}.
    \begin{theorem}\label{thm:conv_unique_if_essential}
        Let $\BspN$ be a Banach space and $\rho$ a strongly continuous and isometric representation of a locally compact abelian group $G$. Then there is a unique $w^*$-continuous (for bounded and tight families) extension of $\bullet_\rho$ to $(M(G), \Vert \cdot \Vert_M)$, i.e., a bounded, bilinear mapping
        $$
        (\mu, S) \mapsto \mu \bullet_\rho S,\qquad \mu \in M(G),\, S \in B,
        $$
        which turns $(B, \Vert \cdot \Vert_B)$ into an essential $L^1(G)$-module.
    \end{theorem} 
    % That the convolution we have defined is a 
    The module action of 
    $L^1(G)$ is obtained by restriction of the action of    $M(G)$. Since $L^1(G)$ does not 
    contain a unit element in the non-discrete case it is interesting to observe that 
   the module $B$ is essential by the following corollary.
    \begin{corollary}[{\cite[Corollary 1]{Feichtinger2022}}]\label{corollary:essential}
    Given the situation in Theorem \ref{thm:existence_convolution} and a bounded approximate unit $(g_\alpha)_{\alpha \in I}$ in $(L^1(G), \Vert \cdot \Vert_{L^1})$, we have that
    $$
    \lim_{\alpha \to \infty} \Vert g_\alpha \bullet_\rho S - S \Vert_B = 0\qquad \textrm{ for all } S \in B.
    $$
    \end{corollary}
    
    For applications, we often want an explicit form of this convolution, which is going to be given in a weak formulation. In the original proof of Theorem \ref{thm:existence_convolution}, the convolution is defined as the limit
    \begin{align}\label{eq:conv_def_bupu}
        \mu \star S = \lim_{|\Psi| \to 0} D_\Psi \mu \bullet_\rho S,
    \end{align}
    where $\Psi$ is a BUPU of $G$, discussed in Section \ref{sec:wiener_and_bupu}, and $D_\Psi : \mu \mapsto \sum_{i\in I} \mu(\psi_i)\delta_{z_i}$ is a \emph{discretization operator}. To compute it, we can use the following lemma which is specific to the Euclidean case.
    \begin{lemma}[{\cite[Lemma 7]{Feichtinger2017}}]\label{lemma:measure_function_convolution_limit}
        Let $f \in C_0(\R^{2d})$, then
        $$
        \lim_{|\Psi| \to 0} D_\Psi \mu (f) = \mu(f) = \int_{\R^{2d}} f(z)\, d\mu(z).
        $$
    \end{lemma}
    Lastly we cover a theorem which we will  use  in Section \ref{sec:gabor_seq_conv}, concerning the  continuity of the action $\bullet_\rho$. Its formulation is based on tight and bounded nets of measures $(\mu_\alpha)_{\alpha \in I} \subset M(G)$ indexed by some set $I$. We remind the reader that nets are a generalization of sequences and all relevant facts can be found in \cite[Section 7]{Feichtinger2022}. When we say that a net is bounded we mean that it is bounded uniformly in the measure norm and tightness means that for any $\varepsilon > 0$, we can find a function $k \in C_c(G)$ such that $\Vert \mu_\alpha - k \cdot \mu_\alpha \Vert_{M} < \varepsilon$ for all $\alpha \in I$.
    \begin{theorem}[{\cite[Theorem 5]{Feichtinger2022}}]\label{theorem:weak_continuity_for_meas_op_abstract_version}
        Let $\rho$ be a strongly continuous, isometric representation of the locally compact group $G$ on the Banach space $(B, \Vert \cdot \Vert_B)$ and $(\mu_\alpha)_{\alpha\in I}$ a bounded and tight net in $(M(G), \Vert \cdot \Vert_{M})$ with $\mu_0 = w^*-\lim_{\alpha \to \infty} \mu_\alpha$. Then one has
        \begin{align*}
            \lim_{\alpha\to\infty} \Vert \mu_\alpha \bullet_\rho S - \mu_0 \bullet_\rho S \Vert_B = 0\qquad \text{for all }S \in B.
        \end{align*}
    \end{theorem}

    \section{Definition and properties of measure-operator convolutions}\label{sec:def_prop_measure_operator_conv}
    In this section we first show that the results recapitulated in Section \ref{sec:prelim_convolutions} apply to measures on $\R^{2d}$ and the Schatten $p$-class of operators (for $p < \infty$) on $L^2(\R^d)$ and then deduce elementary properties of the resulting operators, mostly using Theorem \ref{thm:weak_def_of_measure_operator_convolution} below giving the weak action of measure-operator convolutions.
    
    \subsection{Defining measure-operator convolutions for trace-class operators}
    Our main task in this section is to show that the assumptions of Theorem \ref{thm:existence_convolution} are satisfied for the representation
    \begin{align}\label{eq:rep_def}
        z \mapsto \rho(z),\qquad\rho(z)S = \pi(z)S \pi(z)^*    
    \end{align}
    of the time-frequency plane $\R^{2d}$ on the space of trace-class operators $\mathcal{S}^1$. In fact, we have 
    % This is the contents of the following lemma.
    \begin{lemma}
    The mapping $\rho$ defined by \eqref{eq:rep_def} is a strongly continuous and isometric representation of $\R^{2d}$.
    \end{lemma}
    \begin{proof}
        We first verify linearity, identity preservation and associativity of $\rho$ as
        \begin{align*}
            \rho(z)(\alpha S_1 + \beta S_2) &= \pi(z)(\alpha S_1 + \beta S_2)\pi(z)^* = \alpha \rho(z)S_1 + \beta \rho(z)S_2,\\
            \rho(0) S &= \pi(0) S \pi(0)^* = S,\\
            \rho(z_1)\big(\rho(z_2) S\big) &= \pi(z_1) \big(\pi(z_2) S \pi(z_2)^*\big) \pi(z_1)^* = \pi(z_1 z_2) S \pi(z_1 z_2)^* = \rho(z_1 z_2) S.
        \end{align*}
    	The remaining continuity property follows from \cite[Lemma 2.1]{Bekka1990}.
    \end{proof}
    
    We now obtain the following corollary by applying Theorem \ref{thm:existence_convolution} and Theorem \ref{thm:conv_unique_if_essential}.
    \begin{corollary}\label{corollary:uniqueness}
        If we identify the map $\bullet_\rho : \R^{2d} \times \mathcal{S}^1 \to \mathcal{S}^1,\, (z, S) \mapsto \pi(z) S \pi(z)^*$ with $(\delta_z, S) \mapsto \pi(z) S \pi(z)^*$, there exists an extension to $M(\R^{2d}) \times \mathcal{S}^1$ which satisfies
        $$
        \Vert \mu \bullet_\rho S \Vert_{\mathcal{S}^1} \leq \Vert \mu \Vert_M \Vert S \Vert_{\mathcal{S}^1}\qquad \textrm{for all }\mu \in M(\R^{2d})\textrm{ and } S \in \mathcal{S}^1,
        $$
        and makes $\mathcal{S}^1$ a Banach module over $M(\R^{2d})$. Moreover, this extension is unique among extensions which turn $\mathcal{S}^1$ into an essential $L^1(\R^{2d})$-module upon embedding $L^1(\R^{2d})$ in $M(\R^{2d})$.
    \end{corollary}

    From here on we write $\star$ for $\bullet_\rho$ and call it the \emph{measure-operator convolution} or \emph{function-operator convolution} if the measure is absolutely continuous with respect to Lebesgue measure. We also make the definition $S \star \mu := \mu \star S$ so that we can convolve with measures from both sides. 
    
    Next we show how we can get an explicit formula for the weak action of a measure-operator convolution. This is useful for computations and shows that our approach appropriately extends the function-operator convolutions of quantum harmonic analysis defined with Bochner integrals which can be found in \cite{Werner1984, Luef2018, Berge2023}.
    \begin{theorem}\label{thm:weak_def_of_measure_operator_convolution}
        Let $S \in \mathcal{S}^1$ and $\mu \in M(\R^{2d})$. Then
        \begin{align*}
            \big\langle (\mu \star S) \psi, \phi \big\rangle = \int_{\R^{2d}} \big\langle \pi(z) S \pi(z)^* \psi, \phi \big\rangle\, d\mu(z) \qquad \text{ for all }\psi, \phi \in L^2(\R^d).
        \end{align*}
    \end{theorem}
    \begin{proof}
        Firstly, $\mu \star S$ is defined as the limit of a sequence of $\mathcal{S}^1$ operators which converge in the $\mathcal{S}^1$ norm by Theorem \ref{thm:existence_convolution}. In particular, they converge in the operator norm. Hence, using \eqref{eq:conv_def_bupu}, we have
        \begin{align*}
            \big\langle (\mu \star S) \psi, \phi \big\rangle &= \left\langle \left(\lim_{|\Psi| \to 0} D_\Psi \mu \star S\right)\psi, \phi \right\rangle\\
            &= \lim_{|\Psi| \to 0}\big\langle  (D_\Psi \mu \star S) \psi , \phi \big\rangle\\
            &= \lim_{|\Psi| \to 0}\left\langle  \left(\sum_{i \in I} \mu(\psi_i) \delta_{z_i} \star S\right) \psi , \phi \right\rangle\\
            &= \lim_{|\Psi| \to 0} \sum_{i \in I} \mu(\psi_i) \big\langle \pi(z_i) S \pi(z_i)^* \psi , \phi \big\rangle\\
            &= \lim_{|\Psi| \to 0} D_\Psi \mu \left( \big\langle \pi(\cdot) S \pi(\cdot)^* \psi, \phi\big\rangle\right)\\
            &= \int_{\R^{2d}} \big\langle \pi(z) S \pi(z)^*\psi, \phi \big\rangle\, d\mu(z),
        \end{align*}
        where we used Lemma \ref{lemma:measure_function_convolution_limit} in the last step. To justify this we still need to show that $\langle \pi(\cdot) S \pi(\cdot)^* \psi, \phi\rangle \in C_0(\R^{2d})$. The continuity follows from $\pi$ being a representation and $S$ being a bounded operator. For vanishing at infinity, by an application of Cauchy-Schwarz we deduce that it suffices to show that
        $$
        \Vert \pi(z) S \pi(z)^* \psi \Vert = \Vert S \pi(z)^* \psi \Vert \to 0\quad \textrm{as }z \to \infty.
        $$
        Since $S$ is a compact operator, this follows if we can show that $\pi(z)^* \psi$ goes to zero weakly as $z \to \infty$. This is equivalent to the STFT vanishing at infinity by
        $$
        \langle \pi(z)^*\psi, \phi \rangle = \langle \psi, \pi(z) \phi \rangle = V_\phi \psi(z),
        $$
        which is a well-known fact \cite{grochenig_book}.
    \end{proof}
    
    \subsection{Extending to the Schatten classes}
    Having defined measure-operator convolutions for trace-class operators, we wish to extend the definition to $\mathcal{S}^p$ for $1 < p \leq \infty$. The first step in this direction is defining $\mu \star T$ for $T \in \mathcal{S}^\infty$ using Schatten duality in a way that generalizes the function-operator case (see \cite{Werner1984} and \cite[Proposition 4.2]{Luef2018}) as
    \begin{align*}
        \big\langle \mu \star T, S \big\rangle_{\mathcal{S}^\infty,\, \mathcal{S}^1} = \big\langle \mu, S \star \check{T}^* \big\rangle_{M,\, C_b}.
    \end{align*}
    Here $S$ is taken to be in $\mathcal{S}^1$ and $C_b(\R^{2d})$ is the space of continuous bounded functions. The fact that the quantity in the right hand side is finite follows from that operator-operator convolutions map into $C_b(\R^{2d})$ \cite{Werner1984}. To show that $\mu \star T \in \mathcal{S}^\infty$, we estimate
    \begin{align*}
        \Vert \mu \star T \Vert_{\mathcal{S}^\infty} &= \sup_{\Vert S \Vert_{\mathcal{S}^1} = 1} \big|\big\langle \mu \star T, S \big\rangle_{\mathcal{S}^\infty,\, \mathcal{S}^1 }\big|\\
        &=\sup_{\Vert S \Vert_{\mathcal{S}^1} = 1} \big|\big\langle \mu, S \star \check{T}^* \big\rangle_{M,\, C_b} \big|\\
        &\leq \sup_{\Vert S \Vert_{\mathcal{S}^1} = 1} \Vert \mu \Vert_{M} \Vert S \star \check{T}^* \Vert_{L^\infty} \leq \Vert \mu \Vert_{M} \Vert T \Vert_{\mathcal{S}^\infty},
    \end{align*}
    where we used Proposition \ref{prop:op_conv_properties} \ref{item:op_op_conv_interpol} for this last step. By a version of Riesz-Thorin interpolation for Schatten classes \cite[Theorem 2.10]{Simon2005}, we immediately get the following result on the Schatten boundedness.
    \begin{proposition}\label{prop:interpolated_mapping_meas_op}
        Let $\mu \in M(\R^{2d})$ and $S \in \mathcal{S}^p$ for $1 \leq p \leq \infty$, then
        \begin{align*}
            \Vert \mu \star S \Vert_{\mathcal{S}^p} \leq \Vert \mu \Vert_M \Vert S \Vert_{\mathcal{S}^p}.
        \end{align*}
    \end{proposition}

    We can also go further than this and extend the weak formulation from Theorem \ref{thm:weak_def_of_measure_operator_convolution} to $\mathcal{S}^p$ for $1 \leq  p < \infty$. The proof is a standard approximation argument but we write it out for completeness.
    \begin{corollary}\label{cor:meas_op_conv_weak_p}
        Let $S \in \mathcal{S}^p$ for $1 \leq p < \infty$ and $\mu \in M(\R^{2d})$, then
        \begin{align*}
            \big\langle (\mu \star S) \psi, \phi \big\rangle = \int_{\R^{2d}} \big\langle \pi(z)S \pi(z)^* \psi, \phi\big\rangle\, d\mu(z) \qquad \text{ for all }\psi, \phi \in L^2(\R^d).
        \end{align*}
    \end{corollary}
    \begin{proof}  
        We approximate $S \in \mathcal{S}^p$ by a sequence of finite rank operators $S_N$ consisting of the first $N$ terms of the spectral decomposition of $S$ so that $S_N \in \mathcal{S}^1$ and $S_N \to S \in \mathcal{S}^p$ as $N \to \infty$ since $p < \infty$. In particular, $S_N \to S$ in $\mathcal{L}(L^2)$ also in this case.

        By Cauchy-Schwarz, $\langle (\mu \star S) \psi, \phi \rangle$ exists and our goal is to show that it is equal to the integral in the formulation of the corollary. The first step in this direction is noting that $\langle (\mu \star S_N) \psi, \phi \rangle \to \langle (\mu \star S) \psi, \phi \rangle$ as $N \to \infty$ by Cauchy-Schwarz and the fact that $S_N \to S$ in $\mathcal{L}(L^2)$. Next we use this to write
        \begin{align*}
            &\left|\big\langle (\mu \star S)\psi, \phi \big \rangle  - \int_{\R^{2d}} \big\langle \pi(z)S \pi(z)^* \psi, \phi\big\rangle \, d\mu(z)\right|\\
            &\hspace{7mm}= \left| \lim_{N \to \infty} \big\langle (\mu \star S_N)\psi, \phi \big\rangle - \int_{\R^{2d}} \big\langle \pi(z)S \pi(z)^* \psi, \phi \big\rangle \, d\mu(z)  \right|\\
            &\hspace{7mm}\leq \lim_{N \to \infty} \int_{\R^{2d}}\big| \big\langle [\pi(z) S_N \pi(z)^* - \pi(z) S \pi(z)^*] \psi, \phi \big\rangle \big|\, d|\mu|(z)\\
            &\hspace{7mm}= \int_{\R^{2d}}  \lim_{N \to \infty} \big| \big\langle [\pi(z) S_N \pi(z)^* - \pi(z) S \pi(z)^*] \psi, \phi \big\rangle \big|\, d|\mu|(z),
        \end{align*}
        where used Theorem \ref{thm:weak_def_of_measure_operator_convolution} for the inequality and moved the limit inside the integral using the dominated convergence theorem with dominating function $2 \Vert S \Vert_{\mathcal{L}(L^2)} \Vert \psi \Vert_{L^2} \Vert \phi \Vert_{L^2}$ which is integrable by the boundedness of $\mu$. From here our desired conclusion follows as the integrand can be seen to be zero upon employing Cauchy-Schwarz and using elementary properties of $\pi$ as well as the fact that $S_N \to S$ in $\mathcal{L}(L^2)$ as $N \to \infty$.
    \end{proof}
    Lastly for $p=\infty$ we \emph{define} the weak action $\langle (\mu \star S) \psi, \phi \rangle$ to agree with the integral formulations in Theorem \ref{thm:weak_def_of_measure_operator_convolution} and Corollary \ref{cor:meas_op_conv_weak_p} since we cannot make a denseness argument. This is precisely what is done for function-operator convolutions in \cite{Luef2018}.
    
    \subsection{Properties of measure-operator convolutions}
    From the weak formulation of measure-operator convolutions in Corollary \ref{cor:meas_op_conv_weak_p} and the extension to $p = \infty$, we can deduce generalizations of many of the properties from Section \ref{sec:preliminaries}. Specifically we are able to generalize relevant parts of Proposition \ref{prop:op_conv_properties}.
    \begin{proposition}\label{prop:measure_op_conv_properties}
        Let $\mu \in M(\R^{2d})$, $S \in \mathcal{S}^p$, $T \in \mathcal{S}^q$ for $1 \leq p \leq \infty$ where $\frac{1}{p} + \frac{1}{q} = 1$ and $R$ be a compact operator. Then
        \begin{enumerate}[label=(\roman*)]
    		\item \label{item:me_op_pos} $\mu \star S$ is positive if $\mu$ and $S$ are positive,
    		\item \label{item:me_op_adjoint}$(\mu \star S)^* = \bar{\mu} \star S^*$,
    		\item \label{item:me_op_check} $(\mu \star S)\check{\,} = \check{\mu} \star \check{S}$,
    		\item \label{item:me_op_trace}$\tr(\mu \star S) = \mu(\R^{2d}) \tr(S)$ if $p = 1$,
    		\item \label{item:me_op_comp}$\mu \star R$ is compact, 
    		\item \label{item:me_op_ass1}$(\mu \star S) \star T = \mu * (S \star T)$,
                \item \label{item:me_op_ass2}$(\mu * \nu) \star S = \mu  \star (\nu \star S)$.
    	\end{enumerate}
    \end{proposition}

    \begin{proof}
        Item \ref{item:me_op_pos}:\,\,Using the weak formulation we verify
        \begin{align*}
            \big\langle (\mu \star S)\psi, \psi \big\rangle = \int_{\R^{2d}} \big\langle S \pi(z)^*\psi, \pi(z)^* \psi \big\rangle\, d\mu(z) \geq 0,
        \end{align*}
        since the integral of a non-negative function with respect to a positive measure is non-negative.\\
        Item \ref{item:me_op_adjoint}:\,\, We verify the relation weakly as
        \begin{align*}
            \big\langle (\mu \star S)^* \psi, \phi \big\rangle &= \big\langle \psi, (\mu \star S)\phi \big\rangle = \overline{\big\langle (\mu \star S) \phi, \psi \big\rangle}\\
            &= \overline{\int_{\R^{2d}} \langle \pi(z) S \pi(z)^* \phi, \psi \rangle\,d\mu(z)}\\
            &= \int_{\R^{2d}} \overline{\langle \pi(z) S \pi(z)^* \phi, \psi \rangle}\, d\bar{\mu}(z)\\
            &= \int_{\R^{2d}} \langle \pi(z) S^* \pi(z)^* \psi, \phi \rangle\,d\bar{\mu}(z)\\
            &= \big\langle (\bar{\mu} \star S^*)\psi, \phi \big\rangle.
        \end{align*}
        Item \ref{item:me_op_check}:\,\,Using the same method we find
        \begin{align*}
            \langle P(\mu \star S) P\psi, \phi \rangle &= \langle (\mu \star S) P\psi, P\phi \rangle\\
            &= \int_{\R^{2d}} \langle \pi(z)S \pi(z)^* P\psi, P\phi \rangle \, d\mu(z)\\
            &= \int_{\R^{2d}} \langle \pi(-z)PSP \pi(-z)^* \psi, \phi \rangle \, d\mu(z)\\
            &= \left\langle (\check{\mu} \star \check{S}) \psi, \phi \right\rangle.
        \end{align*}
        The calculation in the middle step is detailed in \cite[Lemma 3.2 (5)]{Skrettingland_master_thesis}.\\
        Item \ref{item:me_op_trace}:\,\,We compute
        \begin{align*}
            \tr(\mu \star S) &= \sum_n \int_{\R^{2d}} \langle S \pi(z)^* e_n, \pi(z)^* e_n \rangle \,d\mu(z)\\
            &= \int_{\R^{2d}} \tr(S)\, d\mu(z) = \mu(\R^{2d}) \tr(S),
        \end{align*}
        where the use of Fubini is justified by $\mu \in M(\R^{2d})$ and \cite[Proposition 18.9]{conway2000a}.\\
        Item \ref{item:me_op_comp}:\,\,Since $R$ is compact, it is the limit in operator norm of a sequence of finite rank operators $(R_n)_n \subset \mathcal{S}^1$. For each $n$, $\mu \star R_n$ is compact by Corollary \ref{corollary:uniqueness}, hence
        $$
            \Vert \mu \star R - \mu \star R_n \Vert_{\mathcal{L}(L^2)} \leq \Vert \mu \Vert_M \Vert R-R_n \Vert_{\mathcal{L}(L^2)},
        $$
        where we used Proposition \ref{prop:interpolated_mapping_meas_op} with $p = \infty$. We conclude that $\mu \star R$ is the limit in operator norm of a sequence of compact operators and hence it too is compact.\\
        Item \ref{item:me_op_ass1}:\,\,We verify
        \begin{align*}
            ((\mu \star S) \star T)(z) &= \tr\left( (\mu \star S) \pi(z) \check{T} \pi(z)^* \right)\\
            &= \sum_n \langle (\mu \star S) \pi(z) \check{T} \pi(z)^* e_n, e_n\rangle\\
            &= \sum_n \int_{\R^{2d}} \big\langle \pi(y) S \pi(y)^* \pi(z) \check{T} \pi(z)^* e_n, e_n\big\rangle\, d\mu(y)\\
            &= \int_{\R^{2d}} (S \star T)(z-y) \,d\mu(y)\\
            &= (\mu * (S \star T))(z),
        \end{align*}
        where the use of Fubini is justified by Hölder's inequality for the Schatten $p$-classes of operators.
        \\
        Item \ref{item:me_op_ass2}:\,\,For $p=1$, this follows from the module structure of $\mathcal{S}^1$ over $M(\R^{2d})$ established in Theorem \ref{thm:existence_convolution} while for $1 < p < \infty$ it follows by approximating $S$ by finite rank operators as in the proof of Corollary \ref{cor:meas_op_conv_weak_p}. However for $p = \infty$ we need to do the full calculation of looking at the weak action and expanding both sides. A quick calculation shows that both of these quantities are equal to
        \begin{align*}
            \int_{\R^{2d}} \int_{\R^{2d}} \langle \pi(z)\pi(y) S \pi(y)^* \pi(z)^* \psi, \phi \rangle \,d\mu(z)\,d\nu(y),
        \end{align*}
        which yields the desired conclusion.
    \end{proof}
    Next we show how the convolution property of the Fourier-Wigner transform from \eqref{eq:fourier_wigner_def} carries over to the measure-operator convolution setting.
    \begin{proposition}
    Let $\mu \in M(\R^{2d})$ and $S \in \mathcal{S}^1$, then
    $$
    \mathcal{F}_W(\mu \star S) = \mathcal{F}_\sigma(\mu)\mathcal{F}_W(S).
    $$
    \end{proposition}
    \begin{proof}
        We follow the same arguments as Werner's original proof in \cite{Werner1984}, assuming that $\mathcal{F}_W(S)$ is integrable, by first computing
        \begin{align*}
            \mathcal{F}_W(\mu \star S)(z) &= e^{-\pi i x \cdot \omega} \tr\left(\pi(-z) (\mu \star S) \right)\\
            &= e^{-\pi i x \cdot \omega} \sum_n \langle \pi(-z) (\mu \star S)(e_n), e_n \rangle\\
            &= e^{-\pi i x \cdot \omega} \sum_n \int_{\R^{2d}} \langle \pi(-z)\pi(z')S\pi(z')^*e_n, e_n \rangle\, d\mu(z')\\
            &= e^{-\pi i x \cdot \omega} \int_{\R^{2d}} \tr\big( \pi(-z)\pi(z')S\pi(z')^* \big)\, d\mu(z'),
        \end{align*}
        where we used Theorem \ref{thm:weak_def_of_measure_operator_convolution} to move to an integral formulation and justified the change of order of summation and integration by Fubini using the integrability of $\mathcal{F}_W(S)$. From here we use the relation $\tr\big(\pi(-z)\pi(z') S \pi(z')^*\big) = e^{-2\pi i \sigma(z, z')} \tr(\pi(-z) S)$ which can be verified directly as is done in \cite{Luef2018} to deduce
        \begin{align*}
            \mathcal{F}_W(\mu \star S)(z) &= e^{-\pi i x \cdot \omega} \int_{\R^{2d}} e^{-2\pi i \sigma(z, z')} \tr\big( \pi(-z)S \big) \,d\mu(z')\\
            &= \mathcal{F}_\sigma(\mu)(z)\mathcal{F}_W(S)(z).
        \end{align*}
        The case where $\mathcal{F}_W(S)$ is not integrable can now be deduced by density.
    \end{proof}
    By specializing Theorem \ref{theorem:weak_continuity_for_meas_op_abstract_version} to measure-operator convolutions, we obtain the following corollary.
    \begin{corollary}\label{corollary:measure_weak_implies_s1}
        Let $(\mu_\alpha)_{\alpha \in I}$ be a bounded and tight net in $M(\R^{2d})$ with $\mu_0 = w^*-\lim_{\alpha \to \infty} \mu_\alpha$, then
        \begin{align*}
            \lim_{\alpha\to\infty} \Vert \mu_\alpha \star S - \mu_0 \star S\Vert_{\mathcal{S}^1} = 0\qquad \text{for all }S \in \mathcal{S}^1.
        \end{align*}
    \end{corollary}
    This is a non-trivial extension to the $M(\R^{2d})$ continuity of measure-operator convolutions of Corollary \ref{corollary:uniqueness} which we will need in Section \ref{sec:gabor_seq_conv}.
    
    \section{Applications to lattice convolutions}\label{sec:applications_to_lattices}
    Much of applied harmonic analysis and time-frequency analysis is carried out on lattices. There is therefore inherent value in translating results to this setting. This was the aim of the recent paper \cite{Skrettingland2020} by Skrettingland, which introduced quantum harmonic analysis on lattices. Using our results on measure-operator convolutions, we can deduce several results from \cite{Skrettingland2020} immediately from Corollary \ref{corollary:uniqueness}.
    
    Before diving into the full technical aspects of measure-operator convolutions, we clarify the links to \cite{Skrettingland2020} and \cite{Dorfler2021} which explicitly and implicitly use lattice convolutions. Later we discuss a natural generalization of Gabor frames which sets the stage for a mixed-state analogue of Gabor multipliers that mirror the relation between localization operators and mixed-state localization operators. From this generalized notion of Gabor frames we are also able to deduce a version of the Berezin-Lieb inequality on lattices. Lastly we recall a result from \cite{Feichtinger2022} which will allow us to prove the promised result on approximating localization operators by Gabor multipliers and the continuity of Gabor multipliers in Section \ref{sec:gabor_seq_conv}.

    \subsection{Generalities}
    In \cite{Skrettingland2020}, {\it lattice convolutions} are defined between sequences $c \in \ell^1(\Lambda)$, where $\Lambda$ is a lattice, and operators $S \in \mathcal{S}^p$ as
    \begin{equation} \label{LamTwDef}
    c \star_\Lambda S = \sum_{\lambda \in \Lambda} c(\lambda) \pi(\lambda) S \pi(\lambda)^* = \underbrace{\left( \sum_{\lambda \in \Lambda} c(\lambda) \delta_\lambda \right)}_{=: \mu_c} \star S.
    \end{equation}
    Following standard terminology for the case of ordinary convolution, one could also call this a {\it semi-discrete twisted convolution}, or interpret it as the action of the discrete (here bounded) measure on the operators $S$ via the (extended) version of the convolution. As sequences in $\ell^1(\Lambda)$ can be isometrically embedded in $M(\R^{2d})$ as discrete measures, we immediately get the bound
    \begin{equation} \label{LamTwEst}
        \Vert c \star_\Lambda S \Vert_{\mathcal{S}^p} \leq \Vert c \Vert_{\ell^1(\Lambda)} \Vert S \Vert_{\mathcal{S}^p}
    \end{equation}
    from Proposition \ref{prop:interpolated_mapping_meas_op}. Other properties, such as the commutativity of measure-operator convolutions, allow us to recover results from \cite{Skrettingland2020} without much work.
    
    Gabor multipliers and more generally \emph{mixed-state} Gabor multipliers (following the terminology of \cite{Luef2019}) are defined in \cite{Skrettingland2020} as $\mu_c \star S$ and they appropriately generalize the classical notion of Gabor multipliers on lattices discussed in Section \ref{sec:prelim_frame_theory} in the case where $S$ is rank-one. See also the early work \cite{fe02} on this subject in the context of LCA groups. 
    
    In \cite{Dorfler2021}, quantum harmonic analysis is used to develop tools for investigating systems of functions. A central tool is the \emph{data operator} $S_\mathcal{D}$ associated to the
    (indexed) \emph{data set} $\D = \{ f_n \}_n $ in $L^2(\R^d)$ given by
    $$
    S_\mathcal{D} = \sum_n f_n \otimes f_n.
    $$
    Motivated by the desire to increase the size of the data set for machine learning purposes, the $\Omega$\emph{-augmentation} of $\D$ with respect to a compact set $\Omega \subset \R^d$ is defined as 
    \begin{align}\label{eq:aug_data_set}
        \D_\Omega = \left\{ \frac{\pi(z)}{|\Omega|^{1/2}}f_n : \, f_n \in \D,\, z \in \Omega \right\},    
    \end{align}
    which leads to the augmented data operator
    \begin{align}\label{eq:aug_data_op}
        S_{\mathcal{D}_\Omega} = \frac{1}{|\Omega|}\int_{\Omega} \sum_n \pi(z)f_n \otimes \pi(z)f_n\,dz = \frac{1}{|\Omega|}\chi_\Omega \star S_\D.    
    \end{align}
    Real-world data augmentation in machine learning is of course based on discrete time-frequency shifts in \eqref{eq:aug_data_set} meaning that the set $\Omega$ should be discrete. However, the formulation \eqref{eq:aug_data_op} is inherently continuous since we are integrating over $\Omega$ and can only be realized as a classical function-operator convolution. To solve this problem, we can use measure-operator convolutions to define a version of \eqref{eq:aug_data_op} which works for discrete $\Omega$ as
    $$
    S_{\D_\Omega} = \left(\frac{1}{|\Omega|}\sum_{\omega \in \Omega} \delta_{\omega}\right) \star S_\D.
    $$
    Not all results from \cite{Dorfler2021} carry over directly to this discrete setting but we will treat one result on the eigenvalues of mixed-state Gabor multipliers in particular that \cite{Dorfler2021} relies on, originally proved in \cite{Feichtinger2003} for Gabor multipliers. To do so, we merge the setup in \cite{Dorfler2021} with the one in \cite{Skrettingland2020} and consider the set $\Omega$ restricted to a lattice $\Lambda$.
    
    \subsection{Mixed-state Gabor frames}
    In the same way that the frame condition \eqref{eq:frame_condition_classical} is what is needed to discretize the reconstruction formula \eqref{eq:reconstruction_classic}, we will see that the reconstruction formula \eqref{eq:reconstruction_cohen} only holds in the lattice setting when $(S, \Lambda)$ satisfies a version of the following condition, originally discussed in \cite{Skrettingland2021}.
    \begin{definition}
        A pair $(S, \Lambda)$ where $S \in \mathcal{S}^1$ is positive and $\Lambda$ is a lattice in $\R^{2d}$ is said to be a \emph{mixed-state Gabor frame with frame constants} $A, B$ if 
        \begin{align}\label{eq:weak_tightness}
            A\Vert \psi \Vert^2 \leq \sum_{\lambda \in \Lambda} Q_S(\psi)(\lambda) \leq B\Vert \psi \Vert^2 \qquad \text{ for all }\psi \in L^2(\R^d)
        \end{align}
        for some $A, B > 0$. Moreover, if $A = B$, the frame is said to be \emph{tight} and the common quantity $A= B$ is referred to as the \emph{frame constant}.
    \end{definition}
    Technically, if $S = \sum_n s_n (\varphi_n \otimes \varphi_n)$, the above definition is equivalent to the collection $\big\{ \sqrt{s_n}\pi(\lambda) \varphi_n \big\}_{\lambda, n}$ being an (infinite) \emph{multi-window} Gabor frame, introduced in \cite{Zibulski1997}. This connection is made clear in the following example which shows how tight mixed-state Gabor frames can be constructed as linear combinations of tight Gabor frames.
    \begin{example}
        Let $\Lambda$ be a lattice in $\R^{2d}$ and $(\varphi_n)_n$  in $L^2(\R^d)$ a collection of windows such that $(\varphi_n, \Lambda)$ is a tight frame with frame constant $A$ for each $n$. Then if $(s_n)_n \in \ell^1$ is a convex combination so that
        $$
        S = \sum_n s_n (\varphi_n \otimes \varphi_n)
        $$
        is a positive trace-class operator, the pair $(S, \Lambda)$ is a tight mixed-state Gabor frame with frame constant $A$. Indeed, we verify
        \begin{align*}
            \sum_{\lambda \in \Lambda} Q_S(\psi)(\lambda) &= \sum_{\lambda \in \Lambda} (\psi \otimes \psi) \star \left( \sum_n s_n (\varphi_n \otimes \varphi_n)\,\check{} \right)\\
            &= \sum_n s_n \sum_{\lambda \in \Lambda} |V_{\varphi_n}\psi(\lambda)|^2 = \sum_n s_n A\Vert \psi \Vert^2 = A \Vert \psi \Vert^2,
        \end{align*}
        using the fact that $(\psi \otimes \psi) \star (\varphi \otimes \varphi)\,\check{} = |V_\varphi \psi|^2$, the frame equality and Fubini for the change of order of summation.
    \end{example}

    Our concept of mixed-state Gabor frames can also be seen as a special case of Gabor g-frames, introduced in \cite[Section 5]{Skrettingland2021}. More specifically, $(S, \Lambda)$ is a mixed-state Gabor frame if and only if $\sqrt{S}$ generates a Gabor g-frame with respect to $\Lambda$. As we will not need the full generality of Gabor g-frames in the following, we stick to the notation above when proving the promised reconstruction formula.
    \begin{proposition}\label{prop:approximation_of_identity_lattice}
        Let $(S, \Lambda)$ be a tight mixed-state Gabor frame with frame constant $A$. We then have the reconstruction of the identity 
        \begin{align}\label{eq:reconstruction_formula}
            \sum_{\lambda \in \Lambda} \pi(\lambda) S \pi(\lambda)^*\psi = A\psi\qquad \textrm{ for all }\psi \in L^2(\R^d).
        \end{align}
    \end{proposition}
    \begin{proof}
        We expand $S$ in its singular value decomposition as $S = \sum_n s_n (\varphi \otimes \varphi)$ with $s_n \geq 0$ for all $n$ by the positivity of $S$. The tightness of the mixed state Gabor frame can then be written more explicitly as
        \begin{align}\label{eq:explicit_tightness_cohen}
            A\Vert \psi \Vert^2 = \sum_{\lambda \in \Lambda} \sum_n s_n |\langle \psi, \pi(\lambda)\varphi_n \rangle|^2,
        \end{align}
        while \eqref{eq:reconstruction_formula} can be written as
        \begin{align}\label{eq:reconstruction_explicit}
            A\psi = \sum_{\lambda \in \Lambda} \sum_n s_n \langle \psi, \pi(\lambda) \varphi_n \rangle \pi(\lambda) \varphi_n.
        \end{align}
        Consider the linear mapping $\Theta : L^2(\R^d) \to \ell^2(\Lambda \times \N)$ given by
        $$
        \Theta \psi = \frac{1}{\sqrt{A}}\big(\sqrt{s_n} \langle \psi, \pi(\lambda)\varphi_n \rangle \big)_{\lambda,\, n}.
        $$
        By \eqref{eq:explicit_tightness_cohen}, $\Theta$ is an isometry and hence preserves inner products. We expand $\psi$ in an orthonormal basis $(e_m)_m$ and use this as
        \begin{align*}
            A\psi &= A\sum_m \langle \psi, e_m \rangle e_m\\
            &= A\sum_m \langle \Theta\psi, \Theta e_m \rangle e_m\\
            &= \sum_m \sum_{\lambda \in \Lambda} \sum_n s_n \langle \psi, \pi(\lambda)\varphi_n \rangle \langle \pi(\lambda) \varphi_n, e_m \rangle e_m\\
            &= \sum_{\lambda \in \Lambda} \sum_n s_n \langle \psi, \pi(\lambda) \rangle \pi(\lambda) \varphi_n,
        \end{align*}
        which is precisely \eqref{eq:reconstruction_explicit}.
    \end{proof}
    \begin{remark}
        In the language of \cite{Skrettingland2021}, the above proposition states that the pair $(\sqrt{S}, \sqrt{S})$ generate dual Gabor g-frames.
    \end{remark}
    We can also deduce a lattice analogue of Proposition \ref{prop:op_conv_properties} \ref{item:op_op_conv_explicit} using the tight frame condition.
    \begin{proposition}\label{prop:sum_over_op_conv}
        Let $T,S \in \mathcal{S}^1$ and $\Lambda$ a lattice such that $\big(S, \Lambda\big)$ is a tight mixed-state Gabor frame with frame constant $1$. Then
        $$
        \sum_{\lambda \in \Lambda} T \star \check{S}(\lambda) = \tr(T).
        $$
    \end{proposition}
    \begin{proof}
        Expand $T$ in its singular value decomposition with orthonormal bases $(\xi_n)_n$ and $(\eta_n)_n$ of $L^2(\R^d)$ as $T = \sum_n t_n (\xi_n \otimes \eta_n)$ and note that
        \begin{align*}
            \sum_{\lambda \in \Lambda} T \star \check{S}(\lambda) &= \sum_n t_n \sum_{\lambda \in \Lambda} (\xi_n \otimes \eta_n) \star \check{S}(\lambda)= \sum_n t_n \sum_{\lambda \in \Lambda} Q_S(\xi_n, \eta_n)(\lambda).
        \end{align*}
        To proceed, we need an analogue of the frame identity \eqref{eq:weak_tightness} for $Q_S(\xi_n, \eta_n)$ which can be shown using a polarization argument. A straightforward calculation shows that
        $$
        Q_S(\psi, \phi)(\lambda) = \frac{1}{4}\Big( Q_S(\psi + \phi)(\lambda) - Q_S(\psi - \phi)(\lambda) + i Q_S(\psi + i\phi)(\lambda) - iQ_S(\psi - i\phi)(\lambda)\Big).
        $$
        Hence by summing over all $\lambda \in \Lambda$ on both sides, we obtain
        \begin{align*}
            \sum_{\lambda \in \Lambda} Q_{S}(\psi, \phi)(\lambda) = \frac{1}{4} \Big( \Vert \psi + \phi \Vert^2 - \Vert \psi - \phi \Vert^2 + i \Vert \psi + i\phi \Vert^2 - i\Vert \psi - i\phi \Vert^2\Big) = \langle \psi, \phi\rangle.
        \end{align*}
        We can now finish the computation as
        \begin{align*}
            \sum_{\lambda \in \Lambda} T \star \check{S} (\lambda) = \sum_n t_n \sum_{\lambda \in \Lambda} Q_S(\xi_n, \eta_n)(\lambda) = \sum_n t_n  \langle \xi_n, \eta_n \rangle = \tr(T).
        \end{align*}
    \end{proof}
    Next, we show how the frame properties of $S$ are changed when we convolve it with a well-behaved measure. This generalizes the continuous property discussed in \cite[Section 7.2]{Luef2019_acc}.
    
    \begin{proposition}
        Let $(S, \Lambda)$ be a mixed-state Gabor frame with frame bounds $A, B$ and $\mu_c \in M(\R^{2d})$ a discrete measure of the form $\mu_c = \sum_{\lambda \in \Lambda} c(\lambda) \delta_\lambda$ where $c \in \ell^1(\Lambda)$. Then $(\check{\mu_c} \star S, \Lambda)$ is also a mixed-state Gabor frame with frame bounds $A \sum_{\lambda \in \Lambda} c(\lambda), B\sum_{\lambda \in \Lambda} c(\lambda)$.
    \end{proposition}
    \begin{proof}
        We compute the sum
        \begin{align*}
            \sum_{\lambda \in \Lambda} Q_{\check{\mu_c} \star S}(\psi)(\lambda) &= \sum_{\lambda \in \Lambda} (\psi \otimes \psi) \star (\mu_c \star \check{S})(\lambda) \\
            &= \sum_{\lambda \in \Lambda} \mu_c * Q_S(\psi)(\lambda)\\
            &= \sum_{\lambda \in \Lambda} \sum_{\lambda' \in \Lambda} Q_S(\psi)(\lambda - \lambda') c(\lambda')\\
            &= \left(\sum_{\lambda \in \Lambda} c(\lambda)\right) \left( \sum_{\lambda \in \Lambda} Q_S(\psi)(\lambda) \right),
        \end{align*}
        where we used Proposition \ref{prop:measure_op_conv_properties} \ref{item:me_op_check} for the first equality, Proposition \ref{prop:measure_op_conv_properties} \ref{item:me_op_ass1} for the second and Fubini for the final. The desired conclusion now follows upon substituting into the frame condition \eqref{eq:weak_tightness}.
    \end{proof}
    
    \subsection{Eigenvalues of mixed-state Gabor multipliers}\label{sec:eigenvalues_of_gabor}
    Gabor multipliers or localization operators on lattices were first systematically investigated by Feichtinger and Nowak in \cite{Feichtinger2003} while the generalization of replacing the window by an operator was first discussed briefly in \cite[Section 4.1]{Skrettingland2020} but never named nor investigated further. Following the terminology of \cite{Luef2019}, when we replace the window of a Gabor multiplier with a trace-class operator, we refer to the resulting operator as a \emph{mixed-state} Gabor multiplier. We make this precise in the following definition.
    \begin{definition}
    Given a lattice $\Lambda \subset \R^{2d}$, an operator $S \in \mathcal{L}(L^2(\R^d))$ and a compact subset $\Omega \subset \R^{2d}$, we refer to the operator
    \begin{align}\label{eq:mixed_state_gabor_multiplier_def}
        G_{\Omega, \Lambda}^S = \sum_{\lambda \in \Lambda} \chi_\Omega(\lambda) \delta_\lambda \star S =: \mu_\Omega^\Lambda \star S
    \end{align}
    as a \emph{mixed-state Gabor multiplier} and write $\mu_\Omega^\Lambda = \sum_{\lambda \in \Lambda} \chi_\Omega(\lambda) \delta_\lambda$ for the measure.
    \end{definition} 
    In light of the reconstruction formula in Proposition \ref{prop:approximation_of_identity_lattice} above it makes sense that we will need to place similar conditions on $S$ and $\Lambda$ in order for mixed-state Gabor multipliers to be well behaved. We collect the properties we will need in the following definition.
    \begin{definition}
        Given a lattice $\Lambda$, an operator $S \in \mathcal{L}(L^2(\R^d))$ is said to be a \emph{density operator with respect to} $\Lambda$ if $S$ is positive, trace-class and such that $(S, \Lambda)$ is a tight mixed-state Gabor frame with frame constant $1$.
    \end{definition}
   Now we are ready to state our main theorem on the eigenvalues of mixed-state Gabor multipliers. It generalizes a result in \cite{Feichtinger2003}  by a method 
   similar to the way how the 
  corresponding  result for localization operators in    \cite{Feichtinger2001} 
  \cite{Luef2019_acc} was derived. 
    \begin{theorem}\label{theorem:measure_eigenvalues_almost_1}
        Let $S$ be a density operator with respect to $\Lambda$, let $\Omega \subset \R^{2d}$ be compact and fix $\delta \in (0,1)$. If $\{ \lambda_k^{R\Omega} \}_k$ are the eigenvalues of $G_{R\Omega, \Lambda}^S$, then
        \begin{align*}
            \frac{\# \big\{ k : \lambda_k^{R\Omega} > 1 - \delta \big\}}{|R\Omega \cap \Lambda| \tr(S)} \to 1\quad \text{as }R \to \infty.
	\end{align*}
    \end{theorem}
    In the pathological case where $(R\Omega)_{R > 0}$ does not exhaust $\R^{2d}$ (e.g. $\Omega = \{ 0 \}$), the result is not meaningful and we from here on assume that this is not the case.
    
    Much of the proof and notation is analogous to that in \cite{Luef2019_acc} and before moving to the proof  we state and prove some important lemmata.
    \begin{lemma}\label{lemma:eigenvalue_bounds}
        Let $S$ be a density operator with respect to $\Lambda$ and $\Omega \subset \R^{2d}$ a compact domain. Then the eigenvalues $\big\{ \lambda_k^\Omega \big\}_k$ of $G_{\Omega, \Lambda}^S$ satisfy $0 \leq \lambda_k^\Omega \leq 1$.
    \end{lemma}
    \begin{proof}
        By the positivity of $G_{\Omega, \Lambda}^S$ from Proposition \ref{prop:measure_op_conv_properties} \ref{item:me_op_pos}, the eigenvalues are non-negative and real-valued. For the upper bound we, let $h_k^\Omega$ denote the eigenfunction associated to $\lambda_k^\Omega$, then
        \begin{align*}
            \lambda_k^\Omega &= \big\langle G_{\Omega, \Lambda}^S h_k^\Omega, h_k^\Omega \big\rangle = \left\langle \left( \sum_{\lambda \in \Lambda} \chi_\Omega(\lambda) \delta_\lambda \star S\right) h_k^\Omega, h_k^\Omega \right\rangle\\
            &\leq \sum_{\lambda \in \Lambda} \big\langle \pi(\lambda) S \pi(\lambda)^* h_k^\Omega, h_k^\Omega \big\rangle = \sum_{\lambda \in \Lambda} Q_S(h_k^\Omega)(\lambda) = \big\Vert h_k^\Omega \big\Vert^2 = 1.
        \end{align*}
    \end{proof}
    In the following, the function $\tilde{S} = S \star \check{S} : \R^{2d} \to \R$ will play an important role. We start by proving some of its properties.
    \begin{lemma}\label{lemma:square_trace_expression}
        Let $G_{\Omega, \Lambda}^S$ be a mixed-state Gabor multiplier. Then
        $$
        \tr\big(\big( G_{\Omega, \Lambda}^S \big)^2\big) = \sum_{\lambda \in \Lambda \cap \Omega} \sum_{\lambda' \in \Lambda \cap \Omega}     \tilde{S}(\lambda'-\lambda).
        $$
    \end{lemma}
    \begin{proof}
    Note first that $\tr\big(\big(G_{\Omega, \Lambda}^S\big)^2\big) = G_{\Omega, \Lambda}^S \star \check{G}_{\Omega,\Lambda}^S(0)$ by \eqref{eq:op_op_conv_def} and hence by Proposition \ref{prop:measure_op_conv_properties} \ref{item:me_op_check}, Proposition \ref{prop:measure_op_conv_properties} \ref{item:me_op_ass1} and the symmetry of $\star$,
    \begin{align*}
        \tr\big(\big( G_{\Omega, \Lambda}^S \big)^2\big) &= \big(\mu_\Omega^\Lambda \star S\big) \star \big(\check{\mu}_\Omega^\Lambda \star \check{S}\big)(0)\\
        &= \mu_\Omega^\Lambda * \big(\check{\mu}_\Omega^\Lambda * (S \star \check{S})\big)(0)\\
        &= \int_{\R^{2d}} \check{\mu}_\Omega^\Lambda * \tilde{S}(0-z)\,d\mu_\Omega^\Lambda(z)\\
        &= \sum_{\lambda \in \Omega \cap \Lambda} \check{\mu}_\Omega^\Lambda * \tilde{S}(-\lambda)\\
        &= \sum_{\lambda \in \Omega \cap \Lambda} \int_{\R^{2d}} \tilde{S}(-z-\lambda)\, d\check{\mu}_\Omega^\Lambda(z)\\
        &= \sum_{\lambda \in \Omega \cap \Lambda} \sum_{\lambda' \in \Omega \cap \Lambda} \tilde{S}(\lambda'-\lambda).
    \end{align*}
    \end{proof}
    
    \begin{lemma}
    Let $S$ be a density operator with respect to $\Lambda$. Then $\tilde{S}$ is non-negative and
    $$
    \sum_{\lambda \in \Lambda} \tilde{S}(\lambda) = \tr(S).
    $$
    \end{lemma}
    \begin{proof}
    Non-negativity follows by Proposition \ref{prop:op_conv_properties} \ref{item:op_op_positive} and the sum follows from $(S, \Lambda)$ being a tight mixed-state Gabor frame with frame constant $1$ by Proposition \ref{prop:sum_over_op_conv}.
    \end{proof}
    
    \begin{lemma}\label{lemma:mixed_state_trace}
        Let $S$ be a density operator with respect to $\Lambda$ and $\big\{ \lambda_k^\Omega \big\}_k$ the eigenvalues of $G_{\Omega, \Lambda}^S$ counted with multiplicity. Then
        $$
        \sum_k \lambda_k^\Omega = |\Omega \cap \Lambda|\tr(S).
        $$
    \end{lemma}
    \begin{proof}
    By a theorem of Lidskii \cite{Simon2005}, the sum of the eigenvalues counted with multiplicity is equal to the trace of the operator and so the result follows from Proposition \ref{prop:measure_op_conv_properties} \ref{item:me_op_trace}.
    \end{proof}
    
    The next lemma is our replacement for the approximation of the identity argument used in the corresponding proofs in \cite{Luef2019_acc} and \cite{Abreu2015}.
    \begin{lemma}\label{lemma:approximate_identity_replacement}
    Let $S$ be a density operator with respect to $\Lambda$. Then
    \begin{align*}
        \lim_{R \to \infty}\left|\frac{1}{|R\Omega \cap \Lambda|} \sum_{\lambda \in R\Omega \cap \Lambda} \sum_{\lambda' \in R\Omega \cap \Lambda} \tilde{S}(\lambda'-\lambda) - \tr(S)\right| = 0.
    \end{align*}
    \end{lemma}
    \begin{proof}
    Fix $\varepsilon > 0$ and note that since $\sum_{\lambda \in \Lambda} \tilde{S}(\lambda) = \tr(S)$, we can find a sufficiently large $R_0$ so that
    $$
    \sum_{\lambda \in R_0\Omega \cap \Lambda} \tilde{S}(\lambda) > \tr(S) - \varepsilon.
    $$
    Note also that if $\lambda \in (R-\operatorname{diam}(R_0\Omega))\Omega \cap \Lambda$ then 
    $$
    R_0 \Omega \cap \Lambda \subset \big\{ \lambda - \lambda' : \lambda ' \in R\Omega \cap \Lambda \big\}.
    $$
    Now let $R' = R-\operatorname{diam}(R_0\Omega)$ and write
    \begin{align*}
        \frac{1}{|R\Omega \cap \Lambda|} \sum_{\lambda \in R\Omega \cap \Lambda} \sum_{\lambda' \in R\Omega \cap \Lambda} \tilde{S}(\lambda'-\lambda) &\geq \frac{1}{|R\Omega \cap \Lambda|} \sum_{\lambda \in R'\Omega \cap \Lambda} \sum_{\lambda' \in R \Omega \cap \Lambda} \tilde{S}(\lambda'-\lambda)\\
        &\geq \frac{1}{|R\Omega \cap \Lambda|} \sum_{\lambda \in R'\Omega \cap \Lambda} \sum_{\lambda' \in R_0 \Omega \cap \Lambda} \tilde{S}(\lambda')\\
        &>\frac{|R'\Omega \cap \Lambda|}{|R\Omega \cap \Lambda|}\big(\tr(S)-\varepsilon\big).
    \end{align*}
    Since $\frac{R'}{R} \to 1$ as $R \to \infty$, the quantity can be made arbitrarily close to $\tr(S)$. Hence the desired conclusion follows.
    \end{proof}
    
    \begin{lemma}\label{lemma:eigenvalue_inequality_estimate}
        Let $S$ be a density operator with respect to $\Lambda$. Then for each $\delta \in (0,1)$,
        $$
        \Big| \# \big\{ k : \lambda_k^\Omega > 1-\delta \big\} - |\Omega \cap \Lambda|\tr(S) \Big| \leq \max\left\{ \frac{1}{\delta}, \frac{1}{1-\delta} \right\} \left| \sum_{\lambda \in \Omega \cap \Lambda} \sum_{\lambda' \in \Omega \cap \Lambda} \tilde{S}(\lambda' - \lambda) - |\Omega \cap \Lambda|\tr(S)\right|.
        $$
    \end{lemma}
    \begin{proof}
        Define the operator $H(G_{\Omega, \Lambda}^S)$ using the eigendecomposition $G_{\Omega, \Lambda}^S = \sum_k \lambda_k^\Omega (h_k^\Omega \otimes h_k^\Omega)$ as
        $$
            H\big( G_{\Omega, \Lambda}^S \big) = \sum_k H(\lambda_k^\Omega) (h_k^\Omega \otimes h_k^\Omega) ,\qquad H(t) = \begin{cases}
		        -t  \qquad &\text{if } 0 \leq t \leq 1-\delta,\\
		        1-t \qquad &\text{if } 1-\delta < t \leq 1,
		    \end{cases}
        $$
		which is well defined since $0 \leq \lambda_k^\Omega \leq 1$ by Lemma \ref{lemma:eigenvalue_bounds}. From this it follows that 
		\begin{align*}
			\tr\big(H\big( G_{\Omega, \Lambda}^S \big)\big) &= \sum_k H(\lambda_k^\Omega)= \#\big\{ k : \lambda_k^\Omega > 1-\delta \big\} - |\Omega \cap \Lambda|\tr(S)
		\end{align*}
		by Lemma \ref{lemma:mixed_state_trace} and the definition of $H$. Hence
		\begin{align*}
			\Big| \#\big\{ k : \lambda_k^\Omega > 1-\delta \big\} - |\Omega \cap \Lambda|\tr(S) \Big| = \big|\tr\big(H\big( G_{\Omega, \Lambda}^S \big)\big)\big| &\leq \tr\big(|H|\big( G_{\Omega, \Lambda}^S \big)\big).
		\end{align*}
		The function $H$ can be bounded as
	    $$
	    |H(t)| \leq \max\left\{ \frac{1}{\delta}, \frac{1}{1-\delta} \right\}(t - t^2),
	    $$
		and so it follows that
		\begin{align*}
		    \Big| \#\big\{ k : \lambda_k^\Omega > 1-\delta \big\} - |\Omega \cap \Lambda|\tr(S) \Big| &\leq \max\left\{ \frac{1}{\delta}, \frac{1}{1-\delta} \right\} \tr\Big( G_{\Omega, \Lambda}^S - \big( G_{\Omega, \Lambda}^S \big)^2\Big),
		\end{align*}
		which can be written as in the original statement by Lemma \ref{lemma:square_trace_expression}.
    \end{proof}
    We are now ready to finish the proof.
    \begin{proof}[Proof of Theorem \ref{theorem:measure_eigenvalues_almost_1}]
        Applying Lemma \ref{lemma:eigenvalue_inequality_estimate} to $R\Omega$, we find
        \begin{align*}
            \Bigg| \# \{ k : \lambda_k^{R\Omega} &> 1-\delta \} - |R\Omega \cap \Lambda|\tr(S) \Bigg|\\
            &\leq \max\left\{ \frac{1}{\delta}, \frac{1}{1-\delta} \right\} \left| \sum_{\lambda \in R\Omega \cap \Lambda} \sum_{\lambda' \in R\Omega \cap \Lambda} \tilde{S}(\lambda' - \lambda) - |R\Omega \cap \Lambda|\tr(S)\right|.
        \end{align*}
        The desired conclusion then follows upon dividing by $|R \Omega \cap \Lambda|$ and applying Lemma \ref{lemma:approximate_identity_replacement}.
    \end{proof}
    
    \subsection{A Berezin-Lieb inequality on lattices}\label{sec:berezin_lieb}
    Berezin-Lieb inequalities have been investigated in \cite{Werner1984, Klauder2012, Luef2018_berezin} and recently used in \cite{Dorfler2021} in the continuous setting. As discussed above, in applications we are more interested in lattice formulations which is why we present a version of the Berezin-Lieb inequality for measure-operator convolutions and operator-operator convolutions on lattices below.
    \begin{theorem}
        Let $S \in \mathcal{S}^1$ be positive with $\tr(S) = 1$ and $\Lambda$ be a lattice such that $(S, \Lambda )$ is a mixed-state Gabor frame with upper frame constant $B$. If $T \in \mathcal{S}^1$ is positive and $\Phi$ is a non-negative, convex and continuous function on a domain containing the spectrum of $T$ and the range of $T \star \check{S}$, then
        $$
        \sum_{\lambda \in \Lambda} \Phi \circ (T \star \check{S})(\lambda) \leq B \tr\big(\Phi(T)\big).
        $$
        Similarly, if $\mu_c = \sum_{\lambda \in \Lambda} c(\lambda) \delta_\lambda$ is a positive bounded measure and $\Phi$ a non-negative, convex and continuous function on a domain containing the spectrum of $\mu_c \star S$ and the range of $B c$ restricted to $\Lambda$, then
        $$
        \tr(\Phi(\mu_c \star S)) \leq \sum_{\lambda \in \Lambda} \Phi(B c(\lambda)),
        $$
        where we also have to  assume that $\Phi$ is non-decreasing if $(S, \Lambda)$ is not tight.
    \end{theorem}
    \begin{proof}
        Using the singular value decomposition  $T = \sum_n t_n (\xi_n \otimes \xi_n)$, we find that
        \begin{align*}
            T \star \check{S} (\lambda) &= \sum_n t_n \big\langle \check{S} \pi(\lambda)^* \xi_n, \pi(\lambda)^* \xi_n \big\rangle = \sum_n t_n Q_S(\xi_n)(\lambda) \\
            \implies \Phi \circ (T \star \check{S})(\lambda) &\leq \sum_n \Phi(t_n) Q_S(\xi_n)(\lambda),
        \end{align*}
        using Jensen's inequality. Thus by summing over $\Lambda$ we get
        \begin{align*}
            \sum_{\lambda \in \Lambda} \Phi \circ (T \star \check{S})(\lambda) &\leq \sum_{\lambda \in \Lambda} \sum_n \Phi(t_n) Q_S(\xi_n)(\lambda)\\
            &= \sum_n \Phi(t_n) \sum_{\lambda \in \Lambda} Q_S(\xi_n)(\lambda)\\
            &=B \tr(\Phi(T)).
        \end{align*}
        Next for the measure-operator convolution statement, note that by the positivity of $\mu_c$ and $S$, $\mu_c \star S$ is positive via Proposition \ref{prop:measure_op_conv_properties} \ref{item:me_op_pos} and so we can write its singular value decomposition as $\mu_c \star S = \sum_n \lambda_n (\xi_n \otimes \xi_n)$. Hence
        \begin{align*}
            \tr(\Phi(\mu_c \star S)) &= \tr\left(\sum_n \Phi(\lambda_n) (\xi_n \otimes \xi_n)\right) \\
            &=\sum_n \Phi\big(\langle (\mu_c \star S) \xi_n, \xi_n \rangle \big)\\
            &= \sum_n \Phi\left( \sum_{\lambda \in \Lambda} c(\lambda) \langle S \pi(\lambda)^* \xi_n, \pi(\lambda)^* \xi_n \rangle \right),
        \end{align*}
        where we used that $\lambda_n = \langle (\mu_c \star S) \xi_n, \xi_n \rangle$. Next if $(S, \Lambda)$ is not a tight frame, we use that $\Phi$ is non-decreasing to multiply the argument by $\frac{B}{\sum_{\lambda \in \Lambda} Q_{\check{S}}(\xi_n)(\lambda)}$ which yields
        \begin{align*}
            \tr(\Phi(\mu_c \star S)) &\leq \sum_n \Phi \left( \frac{\sum_{\lambda \in \lambda} B c(\lambda) Q_{\check{S}}(\xi_n)(\lambda) }{\sum_{\lambda \in \Lambda} Q_{\check{S}}(\xi_n)(\lambda) } \right)\\
            &\leq \sum_n \sum_{\lambda \in \Lambda} \Phi(B c(\lambda)) Q_{\check{S}}(\xi_n)(\lambda),
        \end{align*}
        where the last step follows by Jensen's inequality. By switching the order of summation and using that $\{ \pi(\lambda)^* \xi_n \}_n$ is an orthonormal basis of $L^2(\R^d)$ by the unitarity of $\pi$, we deduce that the last quantity can be written as $\sum_{\lambda \in \Lambda} \Phi(B c(\lambda))$ which is what we wished to show.
    \end{proof}
    
    \subsection{Convergence of sequences of Gabor multipliers}\label{sec:gabor_seq_conv}
    In this section, we focus our attention on deducing the consequences of Corollary \ref{corollary:measure_weak_implies_s1}. We have already seen that measure-operator convolutions generalize both Gabor multipliers and localization operators and so our goal will be to apply the corollary to these settings. In Sections \ref{sec:preliminaries} and \ref{sec:def_prop_measure_operator_conv}, we have used the BUPU framework for the bounded and tight nets of measures but we will now turn our attention to the more explicit setting of parameterized lattices. More specifically, we will consider rectangular lattices $\Lambda_{\alpha, \beta} = \alpha \Z^d \times \beta \Z^d \lhd \R^{2d}$ parameterized by $\alpha, \beta > 0$. This setting is prevalent in applications and while our results can easily be extended to larger classes of lattices we settle for this setting in the interest of clarity.

    Given such a rectangular parameterized lattice, we can define the discretization in measure form of a function $m : \R^{2d} \to \C$ as
    \begin{align}\label{eq:discretization_of_mask}
        \mu_{\alpha, \beta}^m = \alpha^d \beta^d \sum_{\lambda \in \Lambda_{\alpha, \beta}} m(\lambda) \delta_\lambda.
    \end{align}
    In order to apply Corollary \ref{corollary:measure_weak_implies_s1}, we will need to prove that the above measures are uniformly bounded in $\alpha$ and $\beta$. An important tool for this will be the following lemma.
    \begin{lemma}\label{lemma:riemann_W_implies_M}
        For any $m \in W(L^\infty, \ell^1)(\R^{2d})$, we can choose $\alpha_0, \beta_0$ such that
        \begin{align*}
            \big\Vert \mu_{\alpha, \beta}^m \big\Vert_{M} \leq C(\alpha_0, \beta_0)\Vert m \Vert_{W(L^\infty, \ell^1)}
        \end{align*}
        for all $\alpha < \alpha_0$ and $\beta < \beta_0$.
    \end{lemma}
    \begin{proof}
        Each point in $\Lambda_{\alpha, \beta}$ may be associated to a half-open square of the form $Q_{n,k}^{\alpha_0, \beta_0} = \big[\alpha_0 n, \alpha_0 (n+1)\big) \times \big[\beta_0 k, \beta_0(k+1)\big)$ as these cover $\R^{2d}$. We label these collections of points as $P_{n,k} = \Lambda_{\alpha, \beta} \cap Q_{n,k}^{\alpha_0, \beta_0}$. It then holds that
        \begin{align*}
            \Vert \mu_{\alpha, \beta}^m \Vert_M = \alpha^d \beta^d\sum_{\lambda \in \Lambda_{\alpha, \beta}} |m(\lambda)| = \alpha^d \beta^d \sum_{n,k} \sum_{\lambda \in P_{n,k}} |m(\lambda)|.
        \end{align*}
        Meanwhile, by the equivalence of norms on $W(L^\infty, \ell^1)(\R^{2d})$ for different BUPU, we have that $\Vert m \Vert_{W(L^\infty, \ell^1)} \geq C(\alpha_0, \beta_0) \sum_{n,k} \Vert m \Vert_{L^\infty(Q_{n,k}^{\alpha_0, \beta_0})}$ and so our desired result will follow if we can show that
        \begin{align*}
            \alpha^d \beta^d \sum_{\lambda \in P_{n,k}} |m(\lambda)| \leq C\Vert m \Vert_{L^\infty(Q_{n,k}^{\alpha_0, \beta_0})}
        \end{align*}
        for some constant $C$. Indeed, each evaluation $|m(\lambda)|$ for $\lambda \in P_{n,k}$ is bounded from above by $\Vert m \Vert_{L^\infty(Q_{n,k}^{\alpha_0, \beta_0})}$ because $P_{n,k} \subset Q_{n,k}^{\alpha_0, \beta_0}$. Moreover, because $\alpha < \alpha_0, \beta < \beta_0$, the quantity
        $$
        \sum_{\lambda \in P_{n,k}} \alpha^d\beta^d = \alpha^d \beta^d\big|\Lambda_{\alpha, \beta} \cap Q_{n,k}^{\alpha_0, \beta_0}\big|
        $$
        is uniformly bounded, which finishes the proof.
    \end{proof}
    
    \subsubsection{Approximating localization operators}
  As a first application, we will show that the discretization \eqref{eq:discretization_of_mask} essentially approaches the original mask for our purposes, when the lattice scaling parameters tends to zero.
    \begin{theorem}\label{theorem:lattice_integral_convergence_riemann}
        Let $m : \R^{2d} \to \C$ be a Riemann-integrable function in $W(L^\infty, \ell^1)(\R^{2d})$ and $S$ a trace-class operator. Then we have the convergence
        \begin{align*}
            \big\Vert \mu_{\alpha, \beta}^m \star S - m \star S \big\Vert_{\mathcal{S}^1} \to 0 \quad \text{ as }\alpha, \beta \to 0.
        \end{align*}
    \end{theorem}
    \begin{proof}
        As a first step we will show that $m$ may be taken to have compact support without loss of generality. Fix $\varepsilon > 0$ and let $K \subset \R^{2d}$ be such that $\Vert m \chi_{K^c}\Vert_{W(L^\infty, \ell^1)} < \varepsilon$. We may then estimate
        \begin{align*}
            \big\Vert \mu_{\alpha, \beta}^m \star S - m \star S \big\Vert_{\mathcal{S}^1} &\leq \big\Vert \mu_{\alpha, \beta}^m \star S - \mu_{\alpha, \beta}^{m \chi_K} \star S \big\Vert_{\mathcal{S}^1} + \big\Vert \mu_{\alpha, \beta}^{m \chi_K} \star S - m \star S \big\Vert_{\mathcal{S}^1}\\
            &\leq \big\Vert \mu_{\alpha, \beta}^{m \chi_{K^c} }\big\Vert_{M} \Vert S \Vert_{\mathcal{S}^1} +  \big\Vert \mu_{\alpha, \beta}^{m \chi_K} \star S - m \star S \big\Vert_{\mathcal{S}^1}\\
            &\leq C\varepsilon\Vert S \Vert_{\mathcal{S}^1}  +  \big\Vert \mu_{\alpha, \beta}^{m \chi_K} \star S - m\chi_K \star S \big\Vert_{\mathcal{S}^1} + \big\Vert m \chi_K \star S - m \star S \big\Vert_{\mathcal{S}^1}\\
            &\leq C\varepsilon\Vert S \Vert_{\mathcal{S}^1} + \big\Vert \mu_{\alpha, \beta}^{m \chi_K} \star S - m\chi_K \star S \big\Vert_{\mathcal{S}^1} + \Vert m\chi_{K^c} \Vert_M \Vert S \Vert_{\mathcal{S}^1} \\
            &\leq 2C\varepsilon\Vert S \Vert_{\mathcal{S}^1} + \big\Vert \mu_{\alpha, \beta}^{m \chi_K} \star S - m\chi_K \star S \big\Vert_{\mathcal{S}^1},
        \end{align*}
        where we used Lemma \ref{lemma:riemann_W_implies_M} and the estimate of Proposition \ref{prop:interpolated_mapping_meas_op} twice. Since $\varepsilon$ was arbitrary, we may now replace $m \chi_K$ by $m$ and assume that $m$ has compact support.
        
        To apply Corollary \ref{corollary:measure_weak_implies_s1} to $\Vert \mu_{\alpha, \beta}^m \star S - m \star S \Vert_{\mathcal{S}^1}$ and get the desired convergence, we need to show that $\mu_{\alpha, \beta}^m \to m$ in the weak-* sense and that the sequence $\big( \mu_{\alpha, \beta}^m \big)_{\alpha, \beta}$ is bounded and tight. Boundedness follows directly from Lemma \ref{lemma:riemann_W_implies_M} while tightness follows from the compactness of the support of $m$ established above.
        
        The desired weak-* convergence can be formulated as that for any $f \in C_b(\R^{2d})$,
        \begin{align*}
            \mu_{\alpha, \beta}^m (f) \to \int_{\supp(m)} m(z) f(z) \,dz\qquad \text{ as }\alpha, \beta \to 0.
        \end{align*}
        We claim that the left-hand side can be realized as a Riemann sum approximating the right hand side. Indeed, from the definition of $\mu_{\alpha, \beta}^m$ we have that
        \begin{align}\label{eq:mu_alpha_applied_to_f_riemann}
            \mu_{\alpha, \beta}^m(f) = \sum_{\lambda \in \Lambda_{\alpha, \beta}} m(\lambda)f(\lambda) \alpha^d \beta^d,
        \end{align}
        and this sum goes over the rectangles $\big[n\alpha, (n+1)\alpha\big) \times \big[ k\beta, (k+1)\beta\big)$ which have area $\alpha^d \beta^d$. Now since $m$ is Riemann integrable and $f$ is continuous and bounded and therefore uniformly bounded on the support of $m$, $m \cdot f$ is also Riemann integrable and \eqref{eq:mu_alpha_applied_to_f_riemann} converges to $\int_{\supp(m)} m(z)f(z)\,dz$ as promised.
    \end{proof}
    As discussed in the preliminaries section, $m \star S$ is a mixed-state localization operator. Moreover, convolving the discretized version of $m$ with $S$ essentially gives a mixed-state Gabor multiplier. To see this, we define Gabor multipliers $G_{m, \alpha, \beta}^\varphi$ based on the lattice parameters as
    \begin{align}\label{eq:gabor_multiplier_def}
        G_{m, \alpha, \beta}^\varphi\psi = |\Lambda_{\alpha, \beta}|\sum_{n, k \in \mathbb{Z}} m(n\alpha, k\beta) \langle \psi, \pi(n\alpha, k\beta)\varphi \rangle \pi(n\alpha, k\beta)\varphi = \mu_{\alpha, \beta}^m \star (\varphi \otimes \varphi)(\psi).
    \end{align}
    Note that this definition differs from \eqref{eq:mixed_state_gabor_multiplier_def} used earlier, even when the convolution is with an arbitrary trace-class operator. The reason is that we incorporate the lattice normalization as we will take the mask to be constant and the lattice parameters to be varying.
    
    By taking $m$ to be Riemann integrable and in $W(L^\infty, \ell^1)(\R^{2d})$, the rank-one case of Theorem \ref{theorem:lattice_integral_convergence_riemann} allows us to use the definition \eqref{eq:gabor_multiplier_def} and deduce that
    \begin{align*}
        \big\Vert G_{m, \alpha, \beta}^\varphi - A_m^\varphi \big\Vert_{\mathcal{S}^1} \to 0\qquad \text{ as }\alpha, \beta \to 0,
    \end{align*}
    since the localization operator $A_m^\varphi$, from \eqref{eq:def_loc_op}, can be written as $A_m^\varphi = m \star (\varphi \otimes \varphi)$. Even more explicitly, we can formulate the following corollary when the mask is simply an indicator function.
    \begin{corollary}
        Let $\Omega \subset \R^{2d}$ a compact Jordan measurable subset, $\varphi \in L^2(\R^d)$ and $G_{m, \alpha, \beta}^\varphi$ as in \eqref{eq:gabor_multiplier_def}. Then
        \begin{align*}
            \big\Vert G_{\Omega, \alpha, \beta}^\varphi - A_\Omega^\varphi \big\Vert_{\mathcal{S}^1} \to 0\qquad \text{ as }\alpha, \beta \to 0.
        \end{align*}
    \end{corollary}
    The proof is essentially immediate upon noting that Jordan measurability of $\Omega$ is equivalent to Riemann integrability of $\chi_\Omega$ and that compact support implies $\chi_\Omega \in W(L^\infty, \ell^1)(\R^{2d})$.

    These results should be contrasted with those \cite[Section 5.9]{Feichtinger2003} where localization operators are approximated by Gabor multipliers in the trace-class norm in a similar way. The largest difference is that our result is valid for arbitrary $L^2$ windows (and of course also operator windows).

    One consequence of this result is that (mixed-state) localization operators are dense in the space of localization operators. In \cite{Bayer2015}, it was shown that localization operators with varying masks are dense in $\mathcal{S}^1$ as long as $V_\varphi \varphi \neq 0$. The generalization of this to mixed-state localization operators was established in \cite{Kiukas2012, Luef2018} where the condition was replaced by $\mathcal{F}_W(S) \neq 0$. Combining these results, we can approximate arbitrary $\mathcal{S}^1$ operators as follows.
    \begin{corollary}
    For any trace-class operator $S$ with $\mathcal{F}_W(S)$ free of zeros, the set
        $$
        \left\{ \mu_{\alpha, \beta}^m \star S : \alpha, \beta > 0, m \in L^1(\R^{2d})\right\}
        $$
        is dense in $\mathcal{S}^1$.
    \end{corollary}
    \begin{proof}
        The result follows from \cite[Theorem 7.6]{Luef2018} and Theorem \ref{theorem:lattice_integral_convergence_riemann} upon noting that $W(L^\infty, \ell^1)(\R^{2d})$ is dense in $L^1(\R^{2d})$.
    \end{proof}

    \subsubsection{Mask, lattice and window continuity}
    Next, we turn our attention to investigating how the mixed-state Gabor multipliers $\mu_{\alpha,\beta}^m \star S$ behave when varying the parameters. Results of this nature have been proved in \cite{fe03-5, Feichtinger2003} but we will be able to prove stronger convergence under stricter conditions on the mask but weaker conditions on the window.
    \begin{theorem}
        Let $(\alpha_n)_n$ and $(\beta_n)_n$ be sequences converging to $\alpha, \beta > 0$ respectively, $(m_n)_n$ a sequence in the Wiener amalgam space $W(C_0, \ell^1)(\R^{d})$ converging to $m$ in the corresponding norm and $(S_n)_n$ a sequence of trace-class operators converging to $S$ in $\mathcal{S}^1$. Then for the associated mixed-state Gabor multipliers we have the convergence
        \begin{align*}
            G_{m_n, \alpha_n, \beta_n}^{S_n} \to G_{m, \alpha, \beta}^{S} \quad \text{ in }\mathcal{S}^1 \text{ as }n \to \infty.
        \end{align*}
    \end{theorem}
    \begin{proof}
        To begin, we will need to prove that $\big( \mu_{\alpha_n, \beta_n}^{m_n} \big)_n$ is a tight and bounded sequence of measures. Boundedness follows from Lemma \ref{lemma:riemann_W_implies_M} as functions in $W(C_0, \ell^1)(\R^{2d})$ clearly are Riemann integrable and in $W(L^\infty, \ell^1)(\R^{2d})$ and we can bound $\Vert m_n \Vert_{W(L^\infty, \ell^1)}$ from above uniformly since $(m_n)_n$ converges in this space. For tightness, we again use Lemma \ref{lemma:riemann_W_implies_M} to relate the $M(\R^{2d})$ and $W(C_0, \ell^1)(\R^{2d})$ norms and use that for any set $K$, the sequence $\big(m_n\big|_{K^c}\big)_n$ converges in $W(C_0, \ell^1)(\R^{2d})$ which can be made arbitrarily small by letting $K$ be large due to the finiteness of $\Vert m \Vert_{W(C_0, \ell^1)}$. 

        With verification out of the way, we next show that we can treat the convergence of $(S_n)_n$ separately. Indeed, by the triangle inequality
        \begin{align}\label{eq:triangle_gabor_multiplier_window}
            \big\Vert G_{m_n, \alpha_n, \beta_n}^{S_n} - G_{m, \alpha, \beta}^{S} \big\Vert_{\mathcal{S}^1} \leq \big\Vert G_{m_n, \alpha_n, \beta_n}^{S_n} - G_{m_n, \alpha_n, \beta_n}^S \big\Vert_{\mathcal{S}^1} + \big\Vert G_{m_n, \alpha_n, \beta_n}^S - G_{m, \alpha, \beta}^S \big\Vert_{\mathcal{S}^1}.
        \end{align}
        Expanding the first term as $\big\Vert \mu_{\alpha_n, \beta_n}^{m_n} \star (S_n - S) \big\Vert_{\mathcal{S}^1}$ and applying the estimate of Corollary \ref{corollary:essential}, we see that this term can be made arbitrarily small uniformly by the boundedness of $\big( \mu_{\alpha_n, \beta_n}^{m_n} \big)_n$. We may hence focus our attention on showing convergence of the last term of \eqref{eq:triangle_gabor_multiplier_window}. As is hopefully expected, this is done by applying Corollary \ref{corollary:measure_weak_implies_s1} to
        \begin{align*}
            \big\Vert G_{m_n, \alpha_n, \beta_n}^S - G_{m, \alpha, \beta}^S \big\Vert_{\mathcal{S}^1} = \big\Vert \mu_{\alpha_n, \beta_n}^{m_n} \star S - \mu_{\alpha, \beta}^m \star S \big\Vert_{\mathcal{S}^1}.
        \end{align*}
        Since we have already shown that $\big( \mu_{\alpha_n, \beta_n}^{m_n} \big)_n$ is a tight and bounded sequence, it only remains to show that for all $f \in C_b(\R^{2d})$,
        \begin{align*}
            \mu_{\alpha_n, \beta_n}^{m_n}(f) \to \mu_{\alpha, \beta}^m(f)\qquad \text{ as }n \to \infty
        \end{align*}
        We treat $m$ first by estimating
        \begin{equation}\label{eq:deep_m_cont_estimate}
            \begin{aligned}
                \big| \mu_{\alpha_n, \beta_n}^{m_n}(f) - \mu_{\alpha, \beta}^m(f) \big| &\leq \big| \mu_{\alpha_n, \beta_n}^{m_n}(f) - \mu_{\alpha_n, \beta_n}^m(f) \big| + \big|\mu_{\alpha_n, \beta_n}^m(f) - \mu_{\alpha, \beta}^m(f) \big|\\
                &\leq \big| \mu_{\alpha_n, \beta_n}^{m_n - m}(f) \big| + \big|\mu_{\alpha_n, \beta_n}^m(f) - \mu_{\alpha, \beta}^m(f) \big|\\
                &\leq C(2\alpha, 2\beta)\Vert m - m_n \Vert_{W(C_0, \ell^1)} \Vert f \Vert_{C_b} + \big|\mu_{\alpha_n, \beta_n}^m(f) - \mu_{\alpha, \beta}^m(f) \big|,
            \end{aligned}
        \end{equation}
        where we for the last step used Lemma \ref{lemma:riemann_W_implies_M} with the fact that for large enough $n$, $\alpha_n \leq 2 \alpha$ and $\beta_n \leq 2 \beta$ since $\alpha_n \to \alpha$ and $\beta_n \to \beta$ as $n \to \infty$. As the first term can be made arbitrarily small by considering sufficiently large $n$, we may restrict our attention to the second term. Here we must exchange $m$ for a function with compact support. This is done by first estimating that $\big|\mu_{\alpha_n, \beta_n}^m(f) - \mu_{\alpha, \beta}^m(f) \big|$ is less than 
        \begin{align}\label{eq:gabor_cont_triangle_compact_supp_mask}
             \big|\mu_{\alpha_n, \beta_n}^m(f) - \mu_{\alpha_n, \beta_n}^{m \chi_K}(f) \big| + \big|\mu_{\alpha_n, \beta_n}^{m \chi_K}(f) - \mu_{\alpha, \beta}^{m \chi_K}(f) \big| + \big|\mu_{\alpha, \beta}^{m \chi_K}(f) - \mu_{\alpha, \beta}^m(f) \big|
        \end{align}
        for any set $K \subset \R^{2d}$ by the triangle inequality. The first and last terms may be estimated in the same way as in \eqref{eq:deep_m_cont_estimate} and upon choosing $K$ large enough so that $\Vert m \chi_{K^c} \Vert_{W(C_0, \ell^1)}$ is sufficiently small, we see that we only need to estimate the middle term of \eqref{eq:gabor_cont_triangle_compact_supp_mask} where we can replace $m \chi_K$ by a new $m$ with compact support.

        Now since $m$ has compact support, the sum
        \begin{align*}
            \big| \mu_{\alpha_n, \beta_n}^m(f) - \mu_{\alpha, \beta}^m(f) \big| \leq \sum_{k,j \in \Z^d} \big|m(k\alpha_n, j\beta_n) f(k\alpha_n, j\beta_n) \alpha_n^d \beta_n^d - m(k\alpha, j\beta) f(k\alpha, j\beta) \alpha^d \beta^d\big|
        \end{align*}
        has a finite number of terms and so to show that it goes to zero as $n \to \infty$, it suffices to show that each term goes to zero as $n \to \infty$. This is however easy to see by the continuity of $m$ and $f$ and consequently, we are done.
    \end{proof}
    For a comparable earlier result in \cite[Theorem 3.3]{fe03-5}, we had to assume $m \in W(C_0, \ell^2)(\R^{2d})$ and as a result obtained $\mathcal{S}^2$ convergence. In our setting, requiring $m \in W(C_0, \ell^1)(\R^{2d})$ is more natural as Corollary \ref{corollary:measure_weak_implies_s1} only works for $\mathcal{S}^1$ convergence which requires the mask to be a bounded measure or integrable. It should also be noted that the continuity assumption also is natural as the lattice setting means that any discontinuities in the mask may be picked up for some parameters but not all which can change the operator norm. In \cite[Theorem 5.6.7]{Feichtinger2003} trace-class convergence was proved but the mask was assumed to be in $S_0(\R^{2d})$ and then windows were not allowed to vary. The relaxing of conditions compared to the earlier results should largely be attributed to the alternative path: instead of  performing the proof at the level of operator symbol we use methods from quantum harmonic analysis and the framework laid out in \cite{Feichtinger2022}. 
    
    In the interest of clarity, we again highlight the rank-one case of Gabor multipliers.
    \begin{corollary}
        Let $(\alpha_n)_n$ and $(\beta_n)_n$ be sequences converging to $\alpha, \beta > 0$ respectively, $(m_n)_n$ a sequence in the Wiener amalgam space $W(C_0, \ell^1)(\R^{d})$ converging to $m$ and $(\varphi_n)_n$ a sequence of $L^2(\R^d)$ windows converging to $\varphi \in L^2(\R^d)$. Then for the associated  Gabor multipliers we have the convergence
        \begin{align*}
            G_{m_n, \alpha_n, \beta_n}^{\varphi_n} \to G_{m, \alpha, \beta}^\varphi \quad \text{ in }\mathcal{S}^1 \text{ as }n \to \infty.
        \end{align*}
    \end{corollary}
%\section{Conflict of interest statement}

%On behalf of all authors, the corresponding author states that there is no conflict of interest.

    \printbibliography
	
\end{document}